\documentclass{amsart}
\usepackage{amsmath, amssymb, amsthm}

\newtheorem{thm}{Theorem}[section]
\newtheorem{cor}[thm]{Corollary}
\newtheorem{lem}[thm]{Lemma}
\newtheorem{pro}[thm]{Proposition}
\theoremstyle{definition}
\newtheorem{dfn}[thm]{Definition}
\newtheorem{exa}[thm]{Example}
\newtheorem{rmk}[thm]{Remark}

\def\can{\textrm{can}}
\def\div{\textrm{div}}
\def\Diff{\textrm{Diff} \, }
\def\F{{\mathcal F}}
\def\hi{{\hat \iota}}
\def\It{{\widetilde I}}
\def\O{{\mathcal O}}
\def\P{{\mathcal P}}
\def\R{{\mathbb R}}
\def\Symp{\textrm{Symp} \, }
\def\wv{{\widehat \varphi}}
\def\Z{{\mathbb Z}}

\title[$C^0$-characterization of symplectic and contact]{$C^0$-characterization of symplectic and contact embeddings and Lagrangian rigidity}

\author{Stefan M\"uller}

\address{Georgia Southern University, Department of Mathematical Sciences, 65 Georgia Ave.\ Room 3008, P.O.\ Box 8093, Statesboro, GA 30460, USA}

\email{smueller@georgiasouthern.edu}



\begin{document}
\thispagestyle{plain}

\begin{abstract}
We present a novel $C^0$-characterization of symplectic embeddings and diffeomorphisms in terms of Lagrangian embeddings.
Our approach is based on the shape invariant, which was discovered by J.-C.~Sikorav and Y.~Eliashberg, intersection theory and the displacement energy of Lagrangian submanifolds, and the fact that non-Lagrangian submanifolds can be displaced immediately.
This characterization gives rise to a new proof of $C^0$-rigidity of symplectic embeddings and diffeomorphisms.
The various manifestations of Lagrangian rigidity that are used in our arguments come from $J$-holomorphic curve methods.
An advantage of our techniques is that they can be adapted to a $C^0$-characterization of contact embeddings and diffeomorphisms in terms of coisotropic (or pre-Lagrangian) embeddings, which in turn leads to a proof of $C^0$-rigidity of contact embeddings and diffeomorphisms.
We give a detailed treatment of the shape invariants of symplectic and contact manifolds, and demonstrate that shape is often a natural language in symplectic and contact topology.
We consider homeomorphisms that preserve shape, and propose a hierarchy of notions of Lagrangian topological submanifold.
Moreover, we discuss shape related necessary and sufficient conditions for symplectic and contact embeddings, and define a symplectic capacity from the shape.
\end{abstract}

\maketitle

\section{Introduction and main results} \label{sec:intro}
Let $(W, \omega)$ be a symplectic manifold of dimension $2 n$.
We assume for simplicity but without loss of generality that $W$ is connected.
One goal of the present paper is to give a proof of the following characterization theorem for symplectic embeddings.
Denote by $B_r^{2 n}$ the open ball of radius $r > 0$ (centered at the origin) in Euclidean space $\R^{2 n}$ with its standard symplectic structure $\omega_0 = \sum_{i = 1}^n dx_i \wedge dy_i$.

\begin{thm} \label{thm:symp-shape-preserving}
An embedding $\varphi \colon B_r^{2 n} \to W$ is symplectic if and only if it preserves the shape invariant.
\end{thm}

See section~\ref{sec:symp-shape} for the definition and for properties of the shape invariant that are needed in the proof, and section~\ref{sec:symp} for the proof and for necessary results concerning intersection and displacement of (non-)Lagrangian submanifolds.
As corollaries, we obtain the well-known $C^0$-rigidity of symplectic embeddings and diffeomorphisms.

\begin{cor} [\cite{ekeland:sth89}] \label{cor:rig-symp-emb}
Let $\varphi_k \colon B_r^{2 n} \to W$ be a sequence of symplectic embeddings that converges uniformly on compact subsets to an embedding $\varphi \colon B_r^{2 n} \to W$.
Then $\varphi$ is symplectic, that is, $\varphi^* \omega = \omega_0$.
\end{cor}

\begin{cor}[\cite{eliashberg:rsc82,eliashberg:tsw87,gromov:pdr86}] \label{cor:rig-symp-diff}
The group $\Symp (W, \omega)$ of symplectic diffeomorphisms is closed in the group $\Diff (W)$ of diffeomorphisms of $W$ in the $C^0$-topology.
That is, if $\varphi_k \colon W \to W$ is a sequence of symplectic diffeomorphisms that converges uniformly on compact subsets to a diffeomorphism $\varphi \colon W \to W$, then $\varphi$ is symplectic.
\end{cor}

All three results imply analogous versions for anti-symplectic and conformally symplectic embeddings and diffeomorphisms on the one hand, and for embeddings and diffeomorphisms that rescale or reverse the shape invariant on the other hand.
See the end of section~\ref{sec:symp} for precise statements and their proofs.

An advantage of our methods is that they adapt to contact embeddings and diffeomorphisms.
Let $(M, \xi)$ be a contact manifold of dimension $2 n - 1$.
We may again assume for simplicity that $M$ is connected.
Denote by $B_r^{2 n - 1}$ the open ball of radius $r > 0$ (centered at the origin) in $\R^{2 n - 1}$ with its standard contact structure $\xi_0 = \ker \alpha_0$, where $\alpha_0 = dz - \sum_{i = 1}^{n - 1} y_i \, dx_i$.
See Remark~\ref{rmk:coorientation} as well as section~\ref{sec:non-exact} for important remarks concerning coorientation related to the following results.

\begin{thm} \label{thm:contact-shape-preserving}
An embedding $\varphi \colon B_r^{2 n - 1} \to M$ is contact if and only if it preserves the (modified) shape invariant.
\end{thm}

This theorem is another main objective of this paper.
See sections~\ref{sec:coisotropic} and \ref{sec:contact-shape} for details concerning the shape invariants of contact manifolds, and section~\ref{sec:contact} for the proof.
In addition to $J$-holomorphic curve methods, much of the proof uses purely contact topological arguments.

\begin{cor} \label{cor:rig-contact-emb}
Let $\varphi_k \colon B_r^{2 n - 1} \to M$ be a sequence of contact embeddings that converges uniformly on compact subsets to an embedding $\varphi \colon B_r^{2 n - 1} \to M$.
Then $\varphi$ is contact, that is, $\varphi_* \xi_0 = \xi$.
\end{cor}

\begin{cor}[\cite{ms:gae14}] \label{cor:rig-contact-diff}
The group $\Diff (M, \xi)$ of contact diffeomorphisms is closed in the group $\Diff (M)$ of diffeomorphisms of $M$ in the $C^0$-topology.
That is, if $\varphi_k \colon M \to M$ is a sequence of contact diffeomorphisms that converges uniformly on compact subsets to a diffeomorphism $\varphi \colon M \to M$, then $\varphi$ is contact, i.e.\ $\varphi_* \xi = \xi$.
If $\xi$ is coorientable, then the group $\Diff_+ (M, \xi)$ of contact diffeomorphisms that in addition preserve coorientation is also $C^0$-closed in the group $\Diff (M)$.
\end{cor}

The first known proof of Corollary~\ref{cor:rig-symp-diff} is due to Y.~Eliashberg \cite{eliashberg:rsc82,eliashberg:tsw87}, and uses the analysis of wave fronts.
Another proof is based on M.~Gromov's alternative and his non-squeezing theorem \cite{gromov:pcs85,gromov:pdr86}.
Later it was realized by I.~Ekeland and H.~Hofer \cite{ekeland:sth89} that symplectic capacities can be used for a $C^0$-characterization of symplectic embeddings, which gives rise to a proof of Corollary~\ref{cor:rig-symp-emb} and thus of Corollary~\ref{cor:rig-symp-diff}.
See any of the monographs \cite{hofer:sih94,mcduff:ist98,mcduff:hcs04} for a summary.
It is also possible to give a proof of Corollary~\ref{cor:rig-symp-diff} based on the transformation law in topological Hamiltonian dynamics and uniqueness of the topological Hamiltonian isotopy that is associated to a topological Hamiltonian function \cite{mueller:ghh07}.

Theorem~\ref{thm:symp-shape-preserving} and Theorem~\ref{thm:contact-shape-preserving} give rise to $C^0$-characterizations of symplectic and contact embeddings and diffeomorphisms in terms of Lagrangian and coisotropic embeddings, respectively, via the shape invariants.
Both theorems, as well as the first corollary to Theorem~\ref{thm:contact-shape-preserving}, are new results.
A proof of Corollary~\ref{cor:rig-contact-diff} using global methods (via Gromov's alternative) is given in the recent paper \cite{ms:gae14}.
Another version of $C^0$-rigidity of contact diffeomorphisms (that also takes into account the conformal factors of the diffeomorphisms) can be found in the article \cite{ms:tcd1}.

Another advantage of our approach via the shape invariant instead of capacities is that it avoids the cumbersome distinction between symplectic and anti-symplectic.
Moreover, a proof of $C^0$-rigidity via symplectic capacities cannot possibly work in the contact setting, since the capacity of the symplectization of a contact manifold is always infinite.
See section~\ref{sec:contact-shape} for details.

The paper is organized as follows.
Section~\ref{sec:symp-shape} introduces the shape invariants of exact symplectic manifolds, and section~\ref{sec:symp} contains the proof of Theorem~\ref{thm:symp-shape-preserving} and its corollaries.
Sections~\ref{sec:contact-shape} and \ref{sec:contact} are the corresponding sections in the contact case, that is, on the contact shapes and on the proof of Theorem~\ref{thm:contact-shape-preserving} and its corollaries, respectively.
Section~\ref{sec:coisotropic} presents a detailed treatment of coisotropic submanifolds (of maximal dimension), including existence and neighborhood theorems;
these are also called pre-Lagrangian submanifolds in the literature.
Section~\ref{sec:strictly-contact} explains an analogous characterization and rigidity of strictly contact embeddings and diffeomorphisms, and section~\ref{sec:non-exact} generalizes the shape invariants to non-exact symplectic manifolds and contact manifolds that are not necessarily coorientable.
Section~\ref{sec:shape-preserving} discusses identical shapes as a necessary and sometimes sufficient condition for the existence of symplectic and contact embeddings, and section~\ref{sec:homeos} is concerned with homeomorphisms that preserve shape.
In section~\ref{sec:lagrangians} we propose several notions of Lagrangian topological submanifold, and finally section~\ref{sec:capacity} defines a symplectic capacity that is built from (a special case of) the shape invariant.

\section{The Sikorav-Eliashberg symplectic shape invariants} \label{sec:symp-shape}
In this section we review the symplectic shape invariants defined and studied in the papers \cite{eliashberg:nio91} and \cite{sikorav:qpp91, sikorav:rsc89}.
The properties of these invariants that are needed to characterize symplectic embeddings in the next section are rather elementary, with the exception of a theorem of J.~C.~Sikorav regarding the shapes of certain products in the cotangent bundle of a torus.
This last result is only needed to distinguish symplectic from anti-symplectic and conformally symplectic embeddings.

Throughout this paper, let $L$ be a closed and connected $n$-dimensional manifold.
Let $(W, \omega)$ be an exact symplectic manifold of dimension $2 n$, and $\lambda$ be a primitive of $\omega$, i.e.\ a one-form so that $d\lambda = \omega$.
An embedding $\iota \colon L \hookrightarrow W$ is called Lagrangian if $\iota^* \omega = 0$.
The cohomology class $[\iota^* \lambda] \in H^1(L, \R)$ is called its $\lambda$-period.

\begin{dfn}[{\cite{eliashberg:nio91}}] \label{dfn:lambda-shape}
Let $\tau \colon H^1 (W, \R) \to H^1 (L, \R)$ be a homomorphism.
The $(\lambda, L, \tau)$-shape of $W$ is the subset $I (W, \lambda, L ,\tau)$ of $H^1 (L, \R)$ that consists of all points $z \in H^1 (L, \R)$ such that there exists a Lagrangian embedding $\iota \colon L \hookrightarrow W$ with $\iota^* = \tau$ and $z = [\iota^* \lambda]$. \qed
\end{dfn}

Note that this set may be empty.
(See Theorem~\ref{thm:rig-lag-hom-class} for an example.)
The shape invariant depends on the choice of primitive one-form $\lambda$ with $d\lambda = \omega$.

\begin{lem}[{\cite{eliashberg:nio91}}] \label{lem:translation}
If $\lambda' = \lambda + \theta$ is another choice of primitive one-form (that is, $d\theta = 0$), then $I (W, \lambda', L ,\tau) = I (W, \lambda, L ,\tau) + \tau ([\theta])$.
\end{lem}

\begin{proof}
If $\iota$ is a Lagrangian embedding with $\iota^* = \tau \colon H^1 (W, \R) \to H^1 (L, \R)$, then $[\iota^* \lambda'] = [\iota^* \lambda] + [\iota^* \theta] = [\iota^* \lambda] + \iota^* [\theta] = [\iota^* \lambda] + \tau ([\theta])$.
\end{proof}

That is, the shape as an invariant of the symplectic structure is only defined up to translation (by an element of the image of the homomorphism $\tau$) in $H^1 (L, \R)$.
This lemma is a first indication why it is necessary to fix the induced homomorphism $\iota^* = \tau$ on the first cohomology groups.
Additional profound consequences appear later in the proofs of Lemma~\ref{lem:irrational}, Theorem~\ref{thm:torus-shape}, and Theorem~\ref{thm:rig-lag-hom-class}.

\begin{dfn}[{\cite{eliashberg:nio91}}] \label{dfn:symp-shape}
The $(L, \tau)$-shape $I (W, \omega, L ,\tau)$ of $W$ is defined to be the shape $I (W, \lambda, L ,\tau)$ for any choice of primitive one-form $\lambda$, defined up to translation.
We usually omit the symplectic structure $\omega$ from the notation when its choice is understood, and write $I (W, L ,\tau)$ for the shape invariant. \qed
\end{dfn}

\begin{lem} \label{lem:open}
The shape $I (W, L, \tau)$ is an open subset of $H^1 (L, \R)$.
More precisely, the subset $I (W, \lambda, L ,\tau)$ is open, where $H^1 (L, \R)$ carries the natural topology as a finite dimensional real vector space, and this property is independent of the choice of primitive one-form $\lambda$.
\end{lem}

\begin{proof}
Let $\lambda$ be a primitive one-form of $\omega$, i.e.\ $d\lambda = \omega$, and let $\iota \colon L \hookrightarrow W$ be a Lagrangian embedding such that $\iota^* = \tau$ and $[\iota^* \lambda] = z \in H^1 (L,\R)$.
By the Weinstein Lagrangian Neighborhood Theorem, a neighborhood $V$ of the image of $\iota$ can be identified with a neighborhood $U$ of (the image of the) zero section $\iota_0 \colon L \hookrightarrow T^* L$ via a symplectic diffeomorphism $\varphi \colon U \to V$ so that $\varphi \circ \iota_0 = \iota$, i.e.\ $\iota$ corresponds to the zero section, and nearby Lagrangian embeddings correspond to graphs of closed one-forms.
This identification corresponds to a translation in $H^1 (L,\R)$ by $\iota_0^* ([\varphi^* \lambda - \lambda_\can])$, which is equal to $z$ since the $\lambda_\can$-period of $\iota$ is zero.
Here $\lambda_\can$ denotes the canonical one-form on $T^*L$.
If $\sigma \colon L \to T^*L$ is a closed one-form, then $\varphi \circ \sigma$ is a Lagrangian embedding with $(\varphi \circ \sigma)^* = (\varphi \circ \iota_0)^* = \iota^* = \tau$, and $[(\varphi \circ \sigma)^* \lambda] = [\sigma^* ((\varphi^* \lambda - \lambda_\can) + \lambda_\can)] = z + [\sigma^* \lambda_\can] = z + [\sigma]$.
\end{proof}

\begin{rmk} \label{rmk:exactness}
A Lagrangian embedding $\iota$ is called exact if the (closed) one-form $\iota^* \lambda$ is exact.
If $\tau$ is an isomorphism, then any Lagrangian embedding $\iota$ with $\iota^* = \tau$ is exact with respect to the proper choice of primitive of $\omega$, namely $\lambda' = \lambda - \tau^{-1} ([\iota^* \lambda])$.
More generally, there is a choice $\lambda'$ that makes $\iota$ exact if and only if $[\iota^* \lambda]$ lies in the image of the homomorphism $\tau$, independent of the initial choice of primitive $\lambda$. \qed
\end{rmk}

\begin{rmk}
On the other hand, if $\tau = 0$, then the shape is defined without any freedom of translation.
This is the case for instance when $H^1 (W, \R) = 0$.
When $(W, \omega) = (\R^{2 n}, \omega_0)$, the $\lambda$-period $[\iota^* \lambda]$ is also called the Liouville class \cite{polterovich:ggs01} or the symplectic area class \cite{audin:srl94}; it is independent of the choice of primitive $\lambda$ of $\omega_0$. \qed
\end{rmk}

\begin{rmk}
An open subset of an exact symplectic manifold is again an exact symplectic manifold.
Moreover, every point in an arbitrary symplectic manifold has a neighborhood on which the symplectic form is exact.
Indeed, if a subset $U$ of a (not necessarily exact) symplectic manifold is diffeomorphic to an open ball, then the restriction of the symplectic form to $U$ is exact by the Poincar\'e Lemma.
(By making the set $U$ smaller if necessary, one may also invoke Darboux's Theorem.)
Therefore the restriction of $\omega$ to the image of the embedding $\varphi$ in Theorem~\ref{thm:symp-shape-preserving} and Corollary~\ref{cor:rig-symp-emb} is exact, and we may assume without loss of generality that the symplectic manifold $W$ itself is exact.
In fact, since the statement is local, we may assume that $(W, \omega) = (\R^{2 n}, \omega_0)$.
See section~\ref{sec:non-exact} for the definition of the shape invariants of non-exact symplectic manifolds.
Thus Theorem~\ref{thm:symp-shape-preserving} and Corollary~\ref{cor:rig-symp-emb} make sense for arbitrary symplectic manifolds. \qed
\end{rmk}

In this paper, we are mostly interested in the situation in which $L$ is an $n$-dimensional torus $T^n = \R^n / \Z^n$, and $W$ is an open subset either of $(\R^{2 n}, \omega_0)$ or of the cotangent bundle $T^* T^n = T^n \times \R^n$ with its canonical symplectic structure $\omega_\can = d \lambda_\can$, where $\lambda_\can = \sum_{i = 1}^n p_i \, dq_i$, and where $(q_1, \ldots, q_n)$ and $(p_1, \ldots, p_n)$ denote coordinates on the base $T^n = \R^n / \Z^n$ and the fiber $\R^n$, respectively.
We therefore often omit the manifold $T^n$ from the notation.

\begin{dfn}
The $(\lambda, \tau)$-shape of $W$ is the subset $I (W, \lambda, \tau) = I (W, \lambda, T^n, \tau)$ of $H^1 (T^n, \R)$, and the $\tau$-shape $I (W, \tau) = I (W, \omega, \tau)$ of $W$ is the set $I (W, \lambda, \tau)$ for any primitive one-form $\lambda$ of $\omega$, defined up to translation. \qed
\end{dfn}

\begin{rmk}
We choose the cohomology classes $[dq_1], \ldots, [dq_n]$ as a basis of $H^1 (T^n, \R)$ to identify it with the fiber $\R^n$ of the fibration $T^* T^n \to T^n$. \qed
\end{rmk}

\begin{exa} \label{exa:plane}
Let $W$ be a connected open subset of $\R^2$, $\omega_0$ be the standard symplectic form on $\R^2$, and $\lambda$ be a one-form on $\R^2$ with $d\lambda = \omega_0$.
Then every embedding $\iota \colon S^1 \hookrightarrow W$ is Lagrangian, and $[\iota^* \lambda] \in \R = H^1 (S^1, \R)$ is equal to $\pm$ the area enclosed by the image of $\iota$ (with respect to the area form $\omega_0$ and the standard identification of $H^1 (S^1, \R)$ with $\R$, and where the $\pm$ sign depends on whether the orientation of the image of $\iota$ agrees with its orientation as the boundary of the enclosed domain).
Thus $I (W, 0) = (-a, 0) \cup (0, a) \subset \R$, where $a$ is the area of $W$ (which may be infinite).
As remarked above, this shape is independent of the choice of primitive one-form $\lambda$ (i.e.\ there is no freedom of translation).

Suppose that $H_1 (W, \R)$ is non-trivial, and let $\iota \colon S^1 \hookrightarrow W$ be an embedding that represents a generator of $H_1 (W, \R)$.
In other words, the complement of $W$ has at least one bounded component (which are all contractible since $W$ is connected), and the image of $\iota$ is homologous to a union of boundaries of such components.
Then $I (W, \lambda, \iota^*) = (a, a + b)$ or $(- a - b, - a) \subset \R$ (since the orientation is fixed by $\iota^*$), where $a \ge 0$ is the sum of the areas of the enclosed bounded components, and $b$ is the area of $W$ (possibly $\infty$).
Since this shape is defined only up to translation (by $\tau ([\theta]) = [\iota^* \theta]$, where $\theta$ is a closed one-form on $W$, which is not defined globally on $\R^2$ if the vector $\tau ([\theta])$ is non-zero), the only invariant is the length $b$ of the interval, i.e.\ the area of $W$ or $\infty$ if $W$ is unbounded. \qed
\end{exa}

\begin{rmk}
By the preceding example, an embedding $W_1 \to W_2$, where $W_1$ and $W_2$ are open and connected subset of $\R^2$, preserves shape (see Definition~\ref{dfn:symp-shape-preserving}) if and only if it is area preserving.
Theorem~\ref{thm:symp-shape-preserving} is a generalization to symplectic embeddings in higher dimensions. \qed
\end{rmk}

\begin{exa} \label{exa:annulus}
Let $W$ be a non-contractible open and connected subset of the cotangent bundle $T^* S^1 = S^1 \times \R$, and $\tau = \iota_0^*$, where $\iota_0$ is the inclusion of the zero section.
Again every circle embedding is Lagrangian (for dimension reasons).
Embeddings that represent elements of the same shape are all homologous.
Choose a representative cycle; this corresponds to the choice of a primitive one-form $\lambda$ for $\omega_\can$.
If this choice is the zero section (corresponding to the choice $\lambda_\can$), then $[\iota^* \lambda_\can]$ is the signed area enclosed by the image $\iota (S^1)$ and the zero section.
The translation in $H^1 (S^1, \R)$ induced by a different choice of primitive one-form $\lambda$ is the real number $\int_{S^1 \times 0} \lambda$ obtained by integrating $\lambda$ over the zero section (and $[\iota^* \lambda]$ is the signed area between $\iota (S^1)$ and some cycle that is homologous to the zero section).
The shape $I (W, \iota_0^*)$ is an interval $(a, b)$ (where $a$ and $b$ may be negative and infinite), defined up to translation, and its length $b - a$ gives the area of $W$ (possibly $\infty$) independently of the choice of primitive $\lambda$.
The discussion of the shape corresponding to other homomorphisms $\tau$ is analogous to (a combination of the above with) the one in the previous example. \qed
\end{exa}

\begin{lem} \label{lem:shape-non-empty}
If $W$ is any symplectic manifold and $\tau = 0$, then the shape $I (W, \tau)$ is non-empty.
In fact, the vector $z = (z_1, \ldots, z_n)$ is contained in $I (W, 0)$, provided that its coordinates $z_i$ are all positive and sufficiently small.
\end{lem}

\begin{proof}
By Darboux's Theorem, the inclusion of a split torus $S^1 (r_1) \times \cdots \times S^1 (r_n)$ into $\R^2 \times \cdots \times \R^2 = \R^{2 n}$ gives rise to a Lagrangian embedding into $W$, provided the radii $r_i > 0$ of the circles are sufficiently small.
The $\lambda$-period $[\iota^* \lambda]$ of this embedding is $(\pi r_1^2, \ldots, \pi r_n^2) \in \R^n$ for any primitive one-form $\lambda$ of $\omega$ (since $\tau =0$).
\end{proof}

Recall that $GL (n, \Z)$ denotes the group of unimodular matrices, i.e.\ the group of matrices with integer coefficients and determinant equal to $+ 1$ or $- 1$.
In particular, $GL (1, \Z) = \{ \pm 1 \}$ (corresponding to orientation, cf.\ Example~\ref{exa:plane} above).
Every matrix $A \in GL (n, \Z)$ gives rise to a diffeomorphism $A \colon T^n \to T^n$.

\begin{pro} \label{pro:L-diffeos}
If $\phi \colon L \to L$ is a diffeomorphism, then the shape satisfies $I (W, \lambda, L, \phi^* \circ \tau) = \phi^* (I (W, \lambda, L, \tau))$.
In particular, $I (W, \lambda, A \circ \tau) = A (I (W, \lambda, \tau))$ provided that $A \in GL (n, \Z)$.
\end{pro}

\begin{proof}
Let $\iota \colon L \hookrightarrow W$ be a Lagrangian embedding with $\iota^* = \tau$ and $[\iota^* \lambda] = z$.
Then $\iota \circ \phi \colon L \hookrightarrow W$ is again a Lagrangian embedding, with $(\iota \circ \phi)^* = \phi^* \circ \tau$ and $[(\iota \circ \phi)^* \lambda] = \phi^* (z)$.
That proves the inclusion $I (W, \lambda, L, \phi^* \circ \tau) \supset \phi^* (I (W, \lambda, L, \tau))$.
Since $\phi$ is a diffeomorphism, the same argument applies to its inverse (with $\tau$ replaced by $\phi^* \circ \tau$), and thus equality holds.
The last part of the lemma is the special case $L = T^n$ and $\phi = A^t$.
\end{proof}

Recall that a vector $z \in H^1 (L, \R)$ is called rational if the image of $H_1 (L, \Z)$ under the homomorphism $z \colon H_1 (L, \R) \to \R$, $\sigma \mapsto \int_\sigma z$ is a discrete subgroup, and irrational otherwise.

\begin{lem} [Viterbo \cite{sikorav:qpp91}] \label{lem:irrational}
If $W$ is any symplectic manifold of dimension greater than two, then $I (W, 0)$ contains all irrational vectors $z \in H^1 (T^n, \R) = \R^n$.
\end{lem}

\begin{proof}
For a matrix $A \in GL (n,\Z)$, the pre-composition of a Lagrangian embedding $\iota \colon T^n \hookrightarrow W$ with the diffeomorphism $A^t \colon T^n \to T^n$ is again Lagrangian, and $[(\iota \circ A^t)^* \lambda] = A ([\iota^* \lambda])$ (as in the previous lemma).
By Lemma~\ref{lem:shape-non-empty}, $I (W, 0)$ contains all $y$ whose coordinates are positive and sufficiently small.
For every irrational vector $z$ there exists such a vector $y$ and a matrix $A \in GL (n,\Z)$ so that $A y = z$, and the claim follows.
\end{proof}

On the other hand, the shape may not contain all rational vectors.
This is the case for instance if $W$ is a subset of the (symplectic) cylinder $B_r^2 \times \R^{2 n - 2} \subset \R^{2 n}$ by a theorem of Gromov and Sikorav \cite[Theorem~1]{sikorav:qpp91}.

\begin{thm}[\cite{sikorav:qpp91}] \label{thm:sikorav}
If $z$ is rational and the positive generator $\gamma$ of $z (H_1 (L, \Z))$ satisfies $\gamma \ge \pi r^2$, then there exists no Lagrangian embedding $\iota \colon L \hookrightarrow B_r^2 \times \R^{2 n - 2}$ such that $z = [\iota^* \lambda]$.
In other words, $z \notin I (B_r^2 \times \R^{2 n - 2}, \omega_0, L, 0)$.
\end{thm}

\begin{proof}[Sketch of proof]
In order to derive a contradiction, suppose that there exists such a Lagrangian embedding $\iota \colon L \hookrightarrow B_r^2 \times \R^{2 n - 2}$ with $z = [\iota^* \lambda]$.
Then there exists a (non-constant) holomorphic disk $D$ with boundary on the image of $L$ and area $a$ so that $0 < a < \pi r^2$.
But $a = \int_D \omega_0 = \int_{\partial D} \lambda \in z (H_1 (L, \Z))$, and thus $a \ge \gamma$.
\end{proof}

Denote by $\iota_a \colon T^n = T^n \times a \hookrightarrow T^n \times \R^n = T^* T^n$ the canonical embedding.
It is obviously Lagrangian with $[\iota_a^* \lambda_\can] = a$.
It follows immediately from the definition that for $A \subset \R^n$ open and connected, $A \subset I (T^n \times A, \lambda_\can, \iota_0^*)$.
That equality holds is a theorem of Sikorav \cite{sikorav:qpp91}, see also \cite{eliashberg:nio91}.

\begin{thm}[{\cite{eliashberg:nio91}}] \label{thm:torus-shape}
If $A \subset \R^n$ is open and connected, $I (T^n \times A, \lambda_\can, \iota_0^*) = A$.
\end{thm}

\begin{rmk}
We only need this theorem here to distinguish symplectic from anti-symplectic and conformally symplectic embeddings and diffeomorphisms.
Its proof is a simple consequence of Gromov's Theorem on the intersection of the image of an exact Lagrangian embedding into a cotangent bundle with the zero section.
We give the proof here to make it transparent to the reader why the homotopy class of the Lagrangian embedding must be fixed, that is, $\iota^* = \iota_0^*$ in the definition of the shape.
We would like to remark that the only known proofs of Gromov's Theorem use $J$-holomorphic curve techniques \cite{mcduff:ist98, mcduff:hcs04}. \qed
\end{rmk}

\begin{proof}
Let $a = (a_1, \ldots, a_n) \in I (T^n \times A, \lambda_\can, \iota_0^*) \subset \R^n$.
By definition, there exists a Lagrangian embedding $\iota \colon T^n \hookrightarrow T^n \times A$ such that $\iota^* = \iota_0^*$ (as a homomorphism $H^1 (T^n \times A, \R) \to H^1 (T^n, \R)$) and $[\iota^* \lambda_\can] = a$.
The translation $\sigma \colon (q, p) \mapsto (q, p - a)$ in the fiber is a symplectic diffeomorphism that interchanges the one-forms $\lambda_\can$ and $\lambda = \lambda_\can - \sum_{i = 1}^n a_i \, dq_i$, and maps the section $T^n \times a$ to the zero section of $T^* T^n$.
Since the difference $\lambda - \lambda_\can$ is a closed one-form and $\iota^* = \iota_0^*$, the Lagrangian embedding $\sigma \circ \iota \colon T^n \hookrightarrow T^n \times \R^n = T^* T^n$ is exact with respect to the canonical one-form $\lambda_\can$.
Thus by Gromov's Theorem \cite[2.3.B$_4$'']{gromov:pcs85}, see also \cite[Theorem~11.19]{mcduff:ist98} or \cite[Corollary~9.2.15]{mcduff:hcs04}, the image of $\sigma \circ \iota$ intersects the zero section.
Equivalently, the image $\iota (T^n) \subset T^n \times A$ of $\iota$ must intersect $T^n \times a$, and hence $a \in A$.
\end{proof}

\begin{rmk} \label{rmk:torus-shape-hyp}
In our application of Sikorav's Theorem below, we may choose the set $A$ to be contractible (or more generally, to have trivial first and second homotopy groups).
In that case, the full force of Gromov's Theorem is not required, and an alternate and perhaps more elementary proof goes as follows.
It is well-known that Arnold's conjecture holds if $\pi_2 (T^n \times A, T^n) = 0$, see for instance \cite[Section~11.3]{mcduff:ist98} or \cite[Theorem~9.2.14]{mcduff:hcs04}.
(The proof of Arnold's conjecture in this case is due to M.~Chaperon.)
Consider the long exact sequence
	\[ \cdots \rightarrow \pi_2 (T^n) \rightarrow \pi_2 (T^n \times A) \rightarrow \pi_2 (T^n \times A, T^n) \rightarrow \pi_1 (T^n) \stackrel{\! \! \iota_*}{\rightarrow} \pi_1 (T^n \times A) \rightarrow \cdots \]
of homotopy groups of the pair $(T^n \times A, T^n)$.
Since $T^n \times A$ deformation retracts onto $T^n$ times a point, and $\pi_2 (T^n) = 0$, it suffices to show that $\iota_* \colon \pi_1 (T^n) \to \pi_1 (T^n \times A)$ is injective.
But $\pi_1 (T^n) \cong H_1 (T^n, \Z) \cong H^1 (T^n, \Z) \cong \Z^n$, and by hypothesis $\iota^* \colon H^1 (T^n \times A, \R) \to H^1 (T^n, \R)$ is the identity, which implies that $\iota$ induces an isomorphism $\pi_1 (T^n) \to \pi_1 (T^n \times A)$ of the fundamental groups. \qed
\end{rmk}

\begin{pro}[\cite{eliashberg:nio91}] \label{pro:symp-shape}
Let $(W_1, \omega_1)$ and $(W_2, \omega_2)$ be exact symplectic manifolds of the same dimension, and let $\varphi \colon W_1 \to W_2$ be a symplectic embedding.
Then $I (W_1, L, \tau) \subset I (W_2, L, \tau \circ \varphi^*)$ (up to translation).
In fact, if $d\lambda_2 = \omega_2$ and $\lambda_1 = \varphi^* \lambda_2$, then $I (W_1, \lambda_1, L, \tau) \subset I (W_2, \lambda_2, L, \tau \circ \varphi^*)$.
In particular, if $\varphi$ is a symplectic diffeomorphism, then $I (W_1, \lambda_1, L, \tau) = I (W_2, \lambda_2, L, \tau \circ \varphi^*)$.
\end{pro}

\begin{proof}
The statements follow immediately from the definitions by composing every Lagrangian embedding into $W_1$ with the symplectic embedding $\varphi$.
\end{proof}

\begin{rmk}
In the special case $(W_1, \omega_1) = (T^*L, d\lambda_\can)$ and $\tau = \iota_0^*$, the previous proposition gives rise to an alternate proof of Lemma~\ref{lem:open} (choose $\iota$ and $\varphi$ as in the above proof of Lemma~\ref{lem:open}). \qed
\end{rmk}

The proposition implies that the shape is a symplectic invariant, and thus an obstruction to symplectic embedding.
Gromov's proof of the existence of exotic symplectic structures on $\R^{2 n}$ for instance can be restated in terms of the shape.
Recall that for $W = \R^{2 n}$, the homomorphism $\tau$ is automatically trivial, and there is no freedom of translation in the definition of the shape invariant.

\begin{exa} [Gromov \cite{gromov:pcs85}] \label{exa:no-exact}
$I (\R^{2 n}, \omega_0, 0) = \R^n - \{ 0 \}$.
Indeed, embedding split tori (cf.\ the proof of Lemma~\ref{lem:shape-non-empty}) shows that every vector with positive coordinates is contained in $I (\R^{2 n}, \omega_0, 0)$, and thus $\R^n - \{ 0 \} \subset I (\R^{2 n}, \omega_0, 0)$ by Proposition~\ref{pro:L-diffeos}.
On the other hand, Gromov showed that there are no exact Lagrangian embeddings into $(\R^{2 n}, \omega_0)$ \cite[2.3.B$_2$]{gromov:pcs85}, and thus equality holds above.
Gromov also proved that there exist so-called exotic symplectic structures $\omega_{\textrm ex}$ on $\R^{2 n}$ that do admit exact Lagrangian embeddings \cite[2.3.B$_5$]{gromov:pcs85}.
In terms of the shape invariant, this means that $I (\R^{2 n}, \omega_{\textrm ex}, 0)$ contains the zero vector.
The existence of a symplectic embedding $(\R^{2 n}, \omega_{\textrm ex}) \to (\R^{2 n}, \omega_0)$ would therefore contradict Proposition~\ref{pro:symp-shape}. \qed
\end{exa}

\begin{rmk}
Regarding the freedom of translation in the definition of the shape invariant, a statement of the form $I (W_1, L, \tau_1) \subset I (W_2, L, \tau_2)$, without explicit choices of primitive one-forms, will always mean that the inclusion holds up to translation, and likewise for equality.
More precisely, it means that (for every choice of primitive one-form $\lambda_2$ of the symplectic form $\omega_2$) there exists a (corresponding) primitive one-form $\lambda_1$ of $\omega_1$ so that inclusion holds, provided that the shape is computed with respect to these specific one-form(s).

On the other hand, the shape with respect to a specific primitive one-form $\lambda$ (Definition~\ref{dfn:lambda-shape}) is defined as a genuine subset of $H^1 (L, \R)$ (not up to translation), so a statement of the form $I (W_1, \lambda_1, L, \tau_1) \subset I (W_2, \lambda_2, L, \tau_2)$ means genuine inclusion, and likewise for equality. \qed
\end{rmk}

\begin{rmk}
Combining Proposition~\ref{pro:symp-shape} with Lemma~\ref{lem:translation}, we see that if there exists a symplectic embedding $\varphi \colon W_1 \to W_2$, then $I (W_1, \lambda_1, L, \tau) + b$ is a (genuine) subset of $I (W_2, \lambda_2, L, \tau \circ \varphi^*)$, where $b = \tau ([\varphi^* \lambda_2 - \lambda_1])$.
If $W_1 \subset W_2$ and $\varphi$ is isotopic to the identity, then $\tau \circ \varphi^* = \tau$, and if in addition $\varphi$ is Hamiltonian, then the translation vector $b = 0$. \qed
\end{rmk}

The following properties of the symplectic shape invariant are not mentioned in \cite{eliashberg:nio91}, but follow almost directly from the definition, with the exception that Sikorav's Theorem is applied to prove Proposition~\ref{pro:conf-symp-shape}.

\begin{rmk}
More precisely, the proof calls for an explicit computation of the shape of some subset of a given manifold.
We only need the fact that there exists a subset $U \subset W$, an $n$-dimensional manifold $L$, and a homomorphism $\tau$, such that the shape $I (U, L, \tau)$ is not symmetric about the zero vector and not rescaling invariant (up to translation).
At present, Sikorav's Theorem is the only known explicit computation of the shape of an open subset of a symplectic manifold of dimension greater than two.
(There is a generalization of Sikorav's Theorem to arbitrary cotangent bundles $T^* L$, see section~\ref{sec:shape-preserving}, but the statement is most natural for tori since they are parallelizable.) \qed
\end{rmk}

\begin{lem} \label{lem:local-shape}
Let $V \subset W$ be an open subset.
Then for any sufficiently small numbers $a_i > 0$ and any sufficiently small open and connected neighborhood $A$ of $a = (a_1, \ldots, a_n)$ in $\R^n$, there exists a subset $U \subset V$ and a Lagrangian embedding $\iota \colon T^n \hookrightarrow U$, such that $I (U, \iota^*) = A$.
\end{lem}

\begin{proof}
The existence of the Lagrangian embedding $\iota$ follows directly from Darboux's Theorem.
By Weinstein's Lagrangian Neighborhood Theorem, we may identify a neighborhood of $\iota (T^n)$ in $V$ with a neighborhood of the zero section in the cotangent bundle $T^* T^n$, which contains $U = T^n \times A$ provided that the numbers $a_i > 0$ and the neighborhood $A$ are sufficiently small.
After replacing $\iota$ with $\iota_a$ for some $a \in A$, the claim follows from Theorem~\ref{thm:torus-shape}.
\end{proof}

\begin{dfn} \label{dfn:symp-shape-preserving}
Let $(W_1, \omega_1)$ and $(W_2, \omega_2)$ be exact symplectic manifolds of the same dimension.
We say that an embedding $\varphi \colon W_1 \to W_2$ preserves the shape invariants (or for short, preserves the shape) of two open subsets $U \subset W_1$ and $V \subset W_2$ such that $\overline{U} \subset W_1$ is compact and $\varphi (\overline{U}) \subset V$ if $I (U, L, \tau) \subset I (V, L, \tau \circ \varphi^*)$ for every closed and connected $n$-dimensional manifold $L$ and every homomorphism $\tau \colon H^1 (U, \R) \to H^1 (L, \R)$.
An embedding is said to preserve shape if it preserves the shape of all open subsets $U \subset W_1$ and $V \subset W_2$ as above. \qed
\end{dfn}

\begin{rmk}
By Proposition~\ref{pro:symp-shape}, symplectic embeddings preserve the shape invariant.
In fact, symplectic embeddings preserve the shape of subsets without the compactness assumption.
However, our definition has the advantage that it is preserved by uniform limits, see Proposition~\ref{pro:continuous-symp-shape} below.
The restriction of a shape preserving embedding to an open subset by definition again preserves shape, and symplectic is a pointwise condition.
Then by virtue of Theorem~\ref{thm:symp-shape-preserving}, every shape preserving embedding is symplectic, and thus a fortiori preserves the shape of all subsets regardless of whether the closures of the domains are compact or not.
A discussion of this relationship in dimension two appeared in Example~\ref{exa:plane}.
Thus shape preserving is a generalization to higher dimensions of area preserving that also makes sense for homeomorphisms.
This last remark is elaborated in section~\ref{sec:homeos}. \qed
\end{rmk}

\begin{pro} \label{pro:conf-symp-shape}
Let $(W_1, \omega_1)$ and $(W_2, \omega_2)$ be exact symplectic manifolds of the same dimension, and $\varphi \colon W_1 \to W_2$ be a conformally symplectic embedding, i.e.\ $\varphi^* \omega_2 = c \, \omega_1$ for a constant $c \not= 0$.
Then $\varphi$ preserves shape if and only if $c = 1$.
\end{pro}

\begin{proof}
That symplectic embeddings preserve shape is Proposition~\ref{pro:symp-shape}.
The proof of the converse follows along the same lines, with the exception that $\varphi$ rescales the shape of every subset by the factor $| c |$, and reflects it about the zero vector if $c < 0$, with translation vector $b = \tau ([\varphi^* \lambda_2 - \lambda_1])$.
By Lemma~\ref{lem:local-shape}, for any sufficiently small numbers $a_i > 0$ and any sufficiently small open and connected neighborhood $A$ of $a = (a_1, \ldots, a_n)$ in $\R^n$, there exists a subset $U \subset W_1$ and a homomorphism $\tau$ such that $I (U, \tau) = A$.
Choosing $A$ to be a subset of $\R^n$ that is not symmetric about the origin (up to the above translation) shows that $c > 0$, and taking $A$ so that it is not rescaling invariant (e.g.\ a sufficiently small ball so that the position vector of its center is not parallel to $b$) implies that $c = 1$.
\end{proof}

\begin{rmk} \label{rmk:distinguish-anti-symp}
By the proposition, in contrast to symplectic capacities, the shape invariant is able to distinguish between symplectic and anti-symplectic embeddings.
The underlying key to this fact is that while the diffeomorphism $q \mapsto -q$ of the torus causes a reflection of the shape about the origin in $\R^n$, it also reverses the sign of the homomorphism $\tau$, that is, $I (W, \tau) = - I (W, - \tau)$ (cf.\ Proposition~\ref{pro:L-diffeos}). \qed
\end{rmk}

\begin{rmk}
It is crucial to note that the definition requires an embedding $\varphi$ to preserve the shape of open subsets of its domain and not just the shape of the domain itself, and likewise for open subsets of its target.
In the case $n =1$, this amounts to the difference between being area preserving and merely preserving total area (Example~\ref{exa:plane}).
In general, the induced homomorphism $\iota^*$ on the first cohomology groups, and therefore also the subgroup of possible translations of the shape, both depend on these choices.
As seen above, a small tubular neighborhood $N$ of a (small) Lagrangian torus contains a lot of useful information.
On the other hand, in a contractible Darboux neighborhood $U$, the induced homomorphism $\iota^*$ is trivial, and the shape $I (U, \iota^*)$ contains too many $\lambda$-periods of embedded Lagrangian tori $L$ with $\pi_2 (N, L) \not= 0$, see Lemmas~\ref{lem:shape-non-empty} and \ref{lem:irrational}.
In practice, $U$ will often be a tubular neighborhood of the image $L$ of an embedded Lagrangian torus in $W_1$, $V$ will be a tubular neighborhood of $\varphi (L)$ in $W_2$, and $\iota^*$ will be an isomorphism (in fact, the identity with respect to the usual identifications). \qed
\end{rmk}

One of the key ingredients in the proof of Corollary~\ref{cor:rig-symp-emb} is the following continuity property of the shape invariants.

\begin{pro} \label{pro:continuous-symp-shape}
Let $(W_1, \omega_1)$ and $(W_2, \omega_2)$ be exact symplectic manifolds of the same dimension.
Suppose that $\varphi_k \colon W_1 \to W_2$ is a sequence of embeddings that converges uniformly on compact subsets to an embedding $\varphi \colon W_1 \to W_2$, and that $\varphi_k$ preserves the shape invariants for every $k$.
Then $\varphi$ preserves shape.
\end{pro}

\begin{proof}
Let $U \subset W_1$ and $V \subset W_2$ be open subsets such that $\overline{U} \subset W_1$ is compact and $\varphi (\overline{U}) \subset V$.
Since $\varphi_k$ converges to $\varphi$ uniformly on compact subsets, the image $\varphi_k (\overline{U})$ is contained in $V$ for $k$ sufficiently large.
By hypothesis, $I (U, L, \tau) \subset I (V, L, \tau \circ \varphi_k^*)$, and the latter equals $I (V, L, \tau \circ \varphi^*)$ for large $k$, since then $\varphi_k$ is homotopic (in fact, isotopic) to $\varphi$.
\end{proof}

\begin{rmk} \label{rmk:generalized-shape}
There are other meaningful shape invariants one can define by further restricting the homotopy type of the Lagrangian embeddings that contribute to the shape.
It is sometimes useful to consider only Lagrangian embeddings that induce prescribed homomorphisms on the first and second homotopy groups, or are (weakly) homotopic to a given map, cf.\ Remark~\ref{rmk:arnold-conjecture} and Section~\ref{sec:lagrangians} below.
All of the results in this paper continue to hold for such shape invariants, provided only that the additional assumptions on the Lagrangian embedding depend only on its homotopy type (see for instance the proofs of Theorem~\ref{thm:torus-shape}, Proposition~\ref{pro:symp-shape}, and Proposition~\ref{pro:continuous-symp-shape} in this section).
The shape in Definition~\ref{dfn:symp-shape} we are working with in this paper is sufficient for our purposes, and in order to streamline the exposition of this paper as much as possible, we do not formally define these other shape invariants or restate the corresponding results in this more general context. \qed
\end{rmk}

\section{$C^0$-characterization of symplectic embeddings} \label{sec:symp}
In this section we give proofs of the results concerning symplectic embeddings and diffeomorphisms that are stated in section~\ref{sec:intro}.

\begin{rmk} \label{rmk:dim}
For most of this section, we need to assume that $\dim W > 2$; this dimensional restriction is due to the fact that in dimension two there are ``too many'' Lagrangian submanifolds (namely, every one-dimensional submanifold is automatically Lagrangian).
If $\dim W = 2$, then a symplectic form is just an area form, and an embedding is symplectic if and only if it is area preserving.
Thus in dimension two Theorem~\ref{thm:symp-shape-preserving} holds by Example~\ref{exa:plane} (use Darboux's Theorem as in the proof given below for higher dimensions).
Corollaries~\ref{cor:rig-symp-emb} and \ref{cor:rig-symp-diff} are well-known (and easy to prove) for surfaces. \qed
\end{rmk}

An $n$-dimensional submanifold $L$ of $W$ is Lagrangian if the restriction $\omega |_L = 0$; the image of a Lagrangian embedding is of course a Lagrangian submanifold.
A proof of Theorem~\ref{thm:symp-shape-preserving} cannot possibly work without the next lemma.
It says in essence that by Darboux's Theorem, Lagrangian submanifolds (in fact, embeddings) are abundant enough to distinguish (conformally) symplectic embeddings.
The proof uses nothing other than symplectic linear algebra.

\begin{lem} \label{lem:lagrangian}
Let $(W_1, \omega_1)$ and $(W_2, \omega_2)$ be two symplectic manifolds of the same dimension.
An embedding $\varphi \colon W_1 \to W_2$ is conformally symplectic if and only if it preserves Lagrangian submanifolds.
The latter means that the image $\varphi (L)$ is a Lagrangian submanifold whenever $L$ is Lagrangian.
The same statement holds if one restricts to embedded Lagrangian tori that are contained in elements of any given open cover of $W_1$.
\end{lem}

\begin{proof}
The fact that a conformally symplectic embedding preserves Lagrangian submanifolds is obvious; we will prove the converse.

Let $x \in W$.
By Darboux's Theorem, we may assume that $x$ is the origin in $\R^{2 n}$, and that $\omega_1 = \omega_0 = \sum_{i = 1}^n dx_i \wedge dy_i$.
In these local coordinates,
	\[ \varphi^* \omega_2 = \sum_{i, j = 1}^n \left( f_{i j} \, dx_i \wedge dx_j + g_{i j} \, dx_i \wedge dy_j + h_{i j} \, dy_i \wedge dy_j \right). \]
Any two vectors $v_1$ and $v_2$ that lie in an isotropic subspace of $\R^{2 n}$ can be extended to a basis $\{ v_1, \ldots, v_n \}$ of a Lagrangian subspace, which in turn can be extended to a symplectic basis $\{ v_1, \ldots, v_n, w_1, \ldots, w_n \}$ of $\R^{2 n}$.
Let $S_i$ be a circle in the linear (symplectic) subspace spanned by $v_i$ and $w_i$ that is tangent to $v_i$ at the origin.
Then the product $S_1 \times \ldots \times S_n$ is an embedded Lagrangian torus.
If $v_1 = \partial / \partial x_i$ and $v_2 = \partial / \partial x_j$, then by hypothesis
	\[ 0 = \omega_2 \left( d\varphi (v_1), d\varphi (v_2) \right) = \varphi^* \omega_2 \, (v_1, v_2) = f_{i j} (0), \]
i.e.\ the smooth function $f_{i j}$ vanishes at the origin.
Similarly, with $v_1 = \partial / \partial x_i$ and $v_2 = \partial / \partial y_j$, we obtain $g_{i j} (0) = 0$ for $i \not= j$, and the choice $v_1 = \partial / \partial y_i$ and $v_2 = \partial / \partial y_j$ yields $h_{i j} (0) = 0$.
Moreover, if we let $v_1 = \partial / \partial x_i + \partial / \partial x_j$ and $v_2 = \partial / \partial y_i - \partial / \partial y_j$, we see that $g_{i i} (0) = g_{j j} (0)$.

Since $x \in W$ was arbitrary, we have proved that $\varphi^* \omega = g \, \omega$ for a smooth function $g$ on $W$.
Since $\varphi^* \omega$ is closed and $\dim W_1 > 2$, the function $g$ must be constant.
\end{proof}

From the above proof of the lemma, we see that in one direction the following stronger statement holds.
An $n$-dimensional submanifold $L$ is non-Lagrangian if at least one tangent space $T_x L$ is not a Lagrangian subspace of $T_x M$.

\begin{lem} \label{lem:non-Lagrangian}
Let $(W_1, \omega_1)$ and $(W_2, \omega_2)$ be two symplectic manifolds of the same dimension.
Suppose that $\varphi \colon W_1 \to W_2$ is an embedding that is not conformally symplectic at a point $x \in W_1$, and let $U$ be a neighborhood of $x$.
Then there exists a Lagrangian embedding $\iota \colon T^n \hookrightarrow W_1$ through $x$ whose image $L$ is contained in $U$, and so that $\varphi (L)$ is non-Lagrangian (at the point $\varphi (x)$).
If $\lambda$ is a one-form on $W_1$ with $d\lambda = \omega_1$, and $a = (a_1, \ldots, a_n) \in H^1 (T^n, \R) = \R^n$ with $a_i > 0$ sufficiently small, then we may in addition assume that $[\iota^* \lambda] = a$.
In particular, we may assume that $\iota$ is a rational Lagrangian embedding.
\end{lem}

\begin{rmk}
A similar argument applies to two symplectic structures on the same smooth manifold.
That is, two symplectic structures $\omega$ and $\omega'$ on a smooth manifold $W$ are conformally equivalent, i.e.\ there exists a (necessarily non-zero) constant $c$ such that $\omega' = c \, \omega$, if and only if every Lagrangian submanifold with respect to $\omega$ is also a Lagrangian submanifold with respect to $\omega'$.
The same statement holds if one restricts to embedded Lagrangian tori that are contained in elements of any given open cover of $W$.
The proof is essentially the same as the one for Lemma~\ref{lem:lagrangian}, and thus is omitted.
It is also a direct consequence of Lemma~\ref{lem:lagrangian} by considering overlapping Darboux coordinate charts of the two symplectic forms. \qed
\end{rmk}

To complete the proof of Theorem~\ref{thm:symp-shape-preserving}, we need a result of F.~Laudenbach and Sikorav on immediate displacement of non-Lagrangian submanifolds \cite{laudenbach:hdl94}.
It is needed in the proof of Theorem~\ref{thm:symp-shape-preserving} only for embeddings of half-dimensional tori.

\begin{thm}[\cite{laudenbach:hdl94}] \label{thm:displacement}
Let $(W, \omega)$ be a symplectic manifold of dimension $2 n$, and $L$ be a closed and connected submanifold of dimension $n$.
Assume that $L$ is non-Lagrangian, and that the normal bundle of $L$ in $W$ has a non-vanishing section.
Then there exists a Hamiltonian vector field $X_F$ on $W$ that is nowhere tangent to $L$.
In particular, $L$ can be disjoined from itself by a Hamiltonian diffeomorphism.
\end{thm}

For later reference in section~\ref{sec:contact}, we provide a brief sketch of the proof. 

\begin{proof}[Sketch of proof]
Denote by $E = TL^\perp$ the symplectic orthogonal complement of $TL$, which is isomorphic to the normal bundle of $L$ in $W$.
By hypothesis, the $n$-dimensional bundle $E$ thus has a non-vanishing section.
Laudenbach and Sikorav modify such a given section to a non-vanishing section $X$ of $E$ such that there exists a neighborhood of $L$ without compact subset that is invariant by the flow of $X$.
The latter is equivalent to the existence of a smooth function $F$ defined near $L$ such that $d F (X) > 0$.
Of course $dF (X) = \omega (X_F, X)$, so that its non-vanishing combined with the fact that $X$ is a section of $E = TL^\perp$ implies that $X_F$ is nowhere tangent to $L$.
We refer to the short paper \cite{laudenbach:hdl94} for details.
\end{proof}

\begin{rmk}[\cite{laudenbach:hdl94}]
The conclusion of the preceding theorem is actually quite a bit stronger than just the fact that the submanifold $L$ can be displaced from itself by a Hamiltonian diffeomorphism.
If $U$ and $V$ are arbitrary neighborhoods of $L$ with $\overline{U} \subset V$, $\eta$ is a cut-off function with $\eta = 1$ on $U$ that vanishes outside $V$, and $\epsilon > 0$, then the Hamiltonian vector field of the function $\epsilon \eta F$ is also nowhere tangent to $L$.
That is, the manifold $L$ can be displaced from itself by a $C^\infty$-small Hamiltonian isotopy with support in an arbitrarily small neighborhood of $L$.
In particular, the displacement energy of $L$ is zero.
It follows from compactness of $L$ that given such a Hamiltonian isotopy, a sufficiently small neighborhood of $L$ is also displaced from itself by the same Hamiltonian isotopy. \qed
\end{rmk}

As a consequence, Laudenbach and Sikorav prove the following rigidity theorem for (embedded) Lagrangian submanifolds.

\begin{thm}[\cite{laudenbach:hdl94}] \label{thm:rig-lag}
An (embedded) closed $n$-dimensional submanifold of $(\R^{2 n}, \omega_0)$ that is the uniform limit of (embedded) Lagrangian submanifolds is itself Lagrangian.
\end{thm}

\begin{rmk}
The preceding theorem is already enough to prove Corollary~\ref{cor:rig-symp-emb}:
by shrinking the domain of $\varphi$ if necessary, we may assume that its image is contained in a Darboux chart in $W$.
If $L$ is a Lagrangian submanifold of $B_r^{2 n}$ and $\varphi_k$ is a sequence of symplectic embeddings that converges uniformly on compact subsets to $\varphi$, then $\varphi (L)$ is Lagrangian by Theorem~\ref{thm:rig-lag}.
Then by Lemma~\ref{lem:lagrangian}, the embedding $\varphi$ is conformally symplectic, and since it must be volume preserving, $\varphi$ is either symplectic or anti-symplectic.
A standard argument using orientation if $n$ is odd or increasing dimension by $1$ shows that $\varphi$ must be symplectic \cite[Section~12.2]{mcduff:ist98}. \qed
\end{rmk}

We choose to follow a different argument that proves Corollary~\ref{cor:rig-symp-emb} as a genuine corollary to Theorem~\ref{thm:symp-shape-preserving}, since the latter also gives rise to a $C^0$-characterization of symplectic embeddings and diffeomorphisms.
In addition, the present approach avoids the cumbersome argument needed above to distinguish between symplectic and anti-symplectic and conformally symplectic embeddings.
Since it will be needed shortly, we again provide a brief sketch of the proof of the theorem.

\begin{rmk}
Unless explicitly stated otherwise, a tubular neighborhood $N$ of a submanifold $L$ shall mean a tubular neighborhood (with respect to a fixed auxiliary Riemannian metric) that is open and deformation retracts onto $L$.
The specific choice of Riemannian metric is unimportant.
A compact tubular neighborhood is a compact subset whose interior is an open tubular neighborhood $N$ as above. \qed
\end{rmk}

\begin{proof}[Sketch of proof of Theorem~\ref{thm:rig-lag}]
Let $\iota_k \colon L \hookrightarrow \R^{2 n}$ be a sequence of Lagrangian embeddings that converges uniformly to an embedding $\iota \colon L \hookrightarrow \R^{2 n}$, and suppose that the latter is non-Lagrangian.
For notational convenience, we identify $L$ with the image $\iota (L)$ in $\R^{2 n}$.
We may assume without loss of generality that the normal bundle of $L$ in $\R^{2 n}$ has a non-vanishing section (we only need the argument in the case $L = T^n$; see \cite{laudenbach:hdl94} for the general argument).
Let $N_\epsilon$ be a tubular neighborhood of $L$ in $\R^{2 n}$ and $X_F$ be a Hamiltonian vector field defined on a neighborhood of $N_\epsilon$, so that its Hamiltonian flow displaces $N_\epsilon$ and $\| F \| < \epsilon$; these exist by virtue of Theorem~\ref{thm:displacement}.
For $k$ sufficiently large, the image of the Lagrangian embedding $\iota_k$ is contained in this neighborhood $N_\epsilon$.

By a theorem of Gromov and Sikorav \cite{sikorav:qpp91}, and by the Lagrangian suspension construction on the (double of the) Hamiltonian isotopy $\{ \phi_F^t \circ \iota_k \}$, the displacement energy of $\iota_k$ is at least half the area of a non-constant holomorphic disk in $\R^{2 n + 2}$ with boundary on the Lagrangian suspension; see Chapters~3 and 4 of \cite{polterovich:ggs01}.
The proof then follows by showing that the areas of the holomorphic disks remain bounded from below by a topological invariant of the tubular neighborhood $N_{\epsilon_0}$ for a fixed parameter $\epsilon_0 > 0$.
Choosing $\epsilon < \epsilon_0$ sufficiently small yields a contradiction, and therefore $L$ has to be Lagrangian.
See \cite[page~165]{laudenbach:hdl94} for details.
\end{proof}

The hypothesis regarding uniform convergence in Theorem~\ref{thm:rig-lag} can be replaced by an assumption on the homotopy class of the Lagrangian embeddings.
We provide several versions of that result.
Either one of them can be used to prove Theorem~\ref{thm:symp-shape-preserving}.
An embedding $\iota \colon L \hookrightarrow W$ is non-Lagrangian if $\iota^* \omega \not= 0$, or equivalently, its image is a non-Lagrangian submanifold.

\begin{thm} \label{thm:rig-lag-hom-class}
Let $\iota \colon L \hookrightarrow (\R^{2 n}, \omega_0)$ be a non-Lagrangian embedding.
Then there exists a tubular neighborhood $N$ of $\iota (L)$ that admits no Lagrangian embedding $\jmath \colon L \hookrightarrow N$ so that the homomorphism $\jmath_* \colon H_1 (L, \R) \to H_1 (N, \R)$ is injective (i.e.\ an isomorphism).
In particular, the shape $I (N, L, \iota^*)$ is empty.
\end{thm}

\begin{proof}
The proof is almost verbatim the same as the one by Laudenbach and Sikorav that is sketched above.
Again arguing by contradiction, let $N = N_k$ be a sequence of shrinking tubular neighborhoods of $\iota (L)$ with displacement energies converging to zero, and suppose there exists a sequence of Lagrangian embeddings $\iota_k \colon L \hookrightarrow \R^{2 n}$ so that the image of $\iota_k$ is contained in $N_k$.
The fact that the holomorphic disks have non-trivial boundary in $H_1 (N, \R)$ follows in this case, using the same argument, from the assumption that $(\iota_k)_*$ is injective, and the proof that these areas are bounded from below independent of $k$ is verbatim the same.
\end{proof}

\begin{rmk} \label{rmk:arnold-conjecture}
In the proof of Theorem~\ref{thm:symp-shape-preserving} below, we only need Theorem~\ref{thm:rig-lag-hom-class} with $L = T^n$.
In this case there is a more direct proof that goes as follows.
Choose a compact tubular neighborhood $K$ and an open tubular neighborhood $N \subset K$ of $\iota (T^n)$ that is displaced by a Hamiltonian diffeomorphism whose Hamiltonian is compactly supported in $K$.
Suppose that $\jmath \colon T^n \hookrightarrow N$ is a Lagrangian embedding so that $\jmath_* \colon H_1 (L, \R) \to H_1 (N, \R)$ is injective.
Arguing as in Remark~\ref{rmk:torus-shape-hyp}, the latter implies that $\pi_2 (K, \jmath (T^n)) = 0$.
But by (a known case of) the Arnold conjecture (again see \cite[Section~11.3]{mcduff:ist98} or \cite[Theorem~9.2.14]{mcduff:hcs04}), $\jmath (T^n)$ is non-displaceable.
This contradiction shows that no such Lagrangian embedding can exist.

In fact, this argument extends to arbitrary (closed and connected) manifolds $L$ under additional hypotheses that guarantee that $\pi_2 (K, \jmath (L)) = 0$.
This can be achieved for instance by working with a different shape invariant that further restricts the homotopy type of the Lagrangian embedding, see Remark~\ref{rmk:generalized-shape}. \qed
\end{rmk}

\begin{thm} \label{thm:rig-lag-period}
Let $\iota \colon L \hookrightarrow (\R^{2 n}, \omega_0)$ be a non-Lagrangian embedding, $\lambda$ be a one-form on $\R^{2 n}$ with $d\lambda = \omega_0$, and $z \in H^1 (L, \R)$ be rational.
Then there exists a neighborhood $N$ of $\iota (L)$ that does not contain any Lagrangian embeddings $\jmath \colon L \hookrightarrow N$ such that $[\jmath^*\lambda] = z$.
That is, $z \notin I (N, \lambda, L, \tau)$ for any homomorphism $\tau$.
\end{thm}

\begin{proof}
Recall from Example~\ref{exa:no-exact} that there are no exact Lagrangian embeddings into $(\R^{2 n}, \omega_0)$ \cite{gromov:pcs85}, so $z \not= 0$.
The proof of the theorem is then again almost verbatim the same as for the previous two theorems.
The only exception is that in this case the lower bound for the areas of the holomorphic disks follows directly from Theorem~\ref{thm:sikorav} by Gromov and Sikorav: the area of a holomorphic disk coincides with its symplectic area, which in turn equals integration of the primitive one-form $\lambda$ over the boundary of the curve.
In particular, the area of the disk is contained in the image of $H_1 (L, \Z)$ under the homomorphism $z \colon H_1 (L, \R) \to \R$; see \cite{sikorav:qpp91} for details.
\end{proof}

\begin{rmk}
Note that in the above theorems, $N$ is a tubular neighborhood of $L$ and $\iota^*$ is the identity, so that the homomorphism $\iota^* \colon H^1 (N, \R) \to H^1 (L, \R)$ is non-trivial when $H^1 (L, \R) \not= 0$.
Thus in contrast to $I (\R^{2 n}, L, \iota^*)$, the shape $I (N, L, \iota^*)$ is defined only up to translation.
This however does not cause any difficulties in later applications.
Theorem~\ref{thm:rig-lag-period} does not extend to cohomology classes $z$ that are not rational.
Although the displacement energy is still positive (see the end of Subsection~3.2.G in \cite{polterovich:ggs01} for some references), a lower bound depends on more than just the $\lambda$-period; compare to Lemma~\ref{lem:irrational}. \qed
\end{rmk}

\begin{rmk}
We would like to point out that the cohomology class $[\iota^* \lambda]$ does not appear in Theorem~\ref{thm:rig-lag-hom-class}, while Theorem~\ref{thm:rig-lag-period} on the other hand makes no mention of the homomorphism $\iota^*$ on the first cohomology groups.
Either of these results by itself is sufficient to give a proof of Theorem~\ref{thm:symp-shape-preserving}, which we are now in a position to do.
Note that neither Theorem~\ref{thm:rig-lag-hom-class} nor Theorem~\ref{thm:rig-lag-period} mean that there are no Lagrangian embeddings nearby a given non-Lagrangian embedding; by Darboux's Theorem, there are many such Lagrangian embeddings, but they are in a different homotopy class (in fact, in the trivial class, and their relative second fundamental groups are non-trivial), and have small (or irrational) $\lambda$-periods.
The statements mean that there are no Lagrangian embeddings nearby the given non-Lagrangian embedding of a certain shape (the given induced homomorphism on cohomology) or size (the given rational $\lambda$-period). \qed
\end{rmk}

\begin{proof}[Proof of Theorem~\ref{thm:symp-shape-preserving}]
The fact that a symplectic embedding preserves shape is Proposition~\ref{pro:symp-shape}.
We only need to prove the converse.
By Remark~\ref{rmk:dim}, we may assume that $n > 1$.

If $\varphi$ is conformally symplectic at every point $x \in B_r^{2 n}$, then $\varphi^* \omega = c \, \omega_0$ for a constant $c \not= 0$ since $\varphi^* \omega$ is a closed form.
By Proposition~\ref{pro:conf-symp-shape}, we must have $c = 1$.

Suppose then that the embedding $\varphi \colon B_r^{2 n} \to W$ is not conformally symplectic at $x \in B_r^{2 n}$.
By replacing $B_r^{2 n}$ by a small ball centered at $x$, we may assume that the image of $\varphi$ is contained in a Darboux chart in $W$.
By Lemma~\ref{lem:non-Lagrangian}, there exists a Lagrangian embedding $\iota \colon T^n \hookrightarrow B_r^{2 n}$ so that the composition $\varphi \circ \iota \colon T^n \hookrightarrow \R^{2 n}$ is non-Lagrangian.
Let $z = [\iota^* \lambda]$, where $\lambda$ is a one-form on $B_r^{2 n}$ with $d\lambda = \omega_0$.
In particular, if $U$ is any neighborhood of $\iota (T^n)$ in $B_r^{2 n}$, then the shape $I (U, \iota^*)$ is non-empty.
In fact, by Lemma~\ref{lem:open}, it is open.

By Theorem~\ref{thm:rig-lag-hom-class} there exists a neighborhood $V$ of $(\varphi \circ \iota) (L)$ such that there is no Lagrangian embedding $\jmath \colon T^n \hookrightarrow V$ with $\jmath^* = (\varphi \circ \iota)^* \colon H^1 (V, \R) \to H^1 (T^n, \R)$.
That is, the shape $I (V, (\varphi \circ \iota)^*)$ is empty.
By shrinking $U$ if necessary, we may assume that $\overline{U} \subset B_r^{2 n}$ is compact and $\varphi (\overline{U}) \subset V$.
Thus $\varphi$ does not preserve shape.
\end{proof}

\begin{rmk}
For an alternate argument (that replaces the final paragraph of the preceding proof), observe that we may assume that the cohomology class $z$ is rational, and then Theorem~\ref{thm:rig-lag-period} guarantees the existence of a neighborhood $V$ of $(\varphi \circ \iota) (L)$ so that there exists no Lagrangian embedding $\jmath \colon T^n \hookrightarrow V$ with $[\jmath^* \lambda] = z$.
A different choice of primitive one-form $\lambda$ on $V$ with $d\lambda = \omega$ causes a translation of the shape by a vector $b$.
After a small perturbation of $\iota$, i.e.\ by composing with a $C^1$-small symplectic diffeomorphism (corresponding to a closed one-form in $T^* T^n$), we may assume that $z + b$ is rational.
This modification may affect the size of the (tubular) neighborhood $V$ (since the generator of the group $(z + b) (H_1 (L, \Z))$ may be different in general), but it does not affect the argument or conclusion.
Thus for any homomorphism $\tau$ one can choose $z$ so that $z \notin I (V, \tau \circ \varphi^*)$ (up to the above translation determined by a choice of primitive one-forms).
Again by shrinking $U$ if necessary, we may assume that $\overline{U} \subset B_r^{2 n}$ is compact and $\varphi (\overline{U}) \subset V$.
Thus $\varphi$ does not preserve the shape invariant. \qed
\end{rmk}

\begin{proof}[Proof of Corollary~\ref{cor:rig-symp-emb}]
By Propositions~\ref{pro:symp-shape} and \ref{pro:continuous-symp-shape} in the previous section, the hypotheses imply that the embedding $\varphi$ preserves the shape invariant, and then by Theorem~\ref{thm:symp-shape-preserving}, $\varphi$ is a symplectic embedding.
\end{proof}

We state two further corollaries to Corollary~\ref{cor:rig-symp-emb}.
The first one is a special case of the second one, but is stated separately for emphasis and to divide the proofs.

\begin{cor} \label{cor:rig-anti-symp-emb}
Let $\varphi_k \colon B_r^{2 n} \to W$ be a sequence of anti-symplectic embeddings that converges uniformly on compact subsets to an embedding $\varphi \colon B_r^{2 n} \to W$.
Then $\varphi$ is anti-symplectic, that is, $\varphi^* \omega = - \omega_0$.
\end{cor}

\begin{proof}
We may assume without loss of generality that a neighborhood of the image of $\varphi$ is contained in a Darboux chart.
Then for $k$ sufficiently large, the ball $\varphi_k (B_r^{2 n})$ is contained in the same chart.
Let $i$ denote the anti-symplectic involution induced by the reflection about the origin in $\R^{2 n}$.
The sequence $i \circ \varphi_k$ is symplectic and converges to the embedding $i \circ \varphi$, and thus the conclusion follows from Corollary~\ref{cor:rig-symp-emb}.
Equivalently, one can consider the symplectic embeddings $\varphi_k \circ i$.
\end{proof}

\begin{cor} \label{cor:rig-conf-symp-emb}
Let $\varphi_k \colon B_r^{2 n} \to W$ be a sequence of embeddings that converges to an embedding $\varphi \colon B_r^{2 n} \to W$ uniformly on compact subsets, and suppose that $\varphi_k^* \omega = c_k \, \omega_0$.
Then $\varphi$ is conformally symplectic.
Moreover, the numbers $c_k$ converge to a non-zero constant $c$, and $\varphi^* \omega = c \, \omega_0$.
\end{cor}

\begin{proof}
We may again assume without loss of generality that a neighborhood of the image of $\varphi$ is contained in a Darboux chart in $W$.
Then $\varphi_k (B_r^{2 n})$ is contained in the same Darboux chart for $k$ sufficiently large.
Denote by $m_s$ multiplication by $s \not= 0$ in $\R^{2 n}$.
These conformally symplectic diffeomorphisms depend continuously on the parameter $s$.
The sequence $m_{c_k}^{-1} \circ \varphi_k$ (or the sequence $\varphi_k \circ m_{c_k}^{-1}$) is symplectic, and thus the proof follows from Corollary~\ref{cor:rig-symp-emb} in the same way Corollary~\ref{cor:rig-anti-symp-emb} does, once we show that the numbers $c_k$ form a Cauchy sequence.

Choose a subsequence of $\varphi_k$ such that the numbers $c_k$ all have the same sign.
Composing with the anti-symplectic involution $i$ from the proof of Corollary~\ref{cor:rig-anti-symp-emb} if necessary, we may assume that $c_k > 0$.
Let $r' < r$ be positive.
Since the volume of $\varphi_k (B_{r'}^{2 n})$ is $c_k^n$ times the volume of $B_{r'}^{2 n}$, and the embeddings $\varphi_k$ converge uniformly on (the closure of) $B_{r'}^{2 n}$, the numbers $c_k$ converge to a number $c \not= 0$ (where $c^n$ is the volume of the ball $\varphi (B_{r'}^{2 n})$).
\end{proof}

\begin{proof}[Proof of Corollary~\ref{cor:rig-symp-diff}]
By Darboux's Theorem, a neighborhood of a point $x \in W$ can be identified with a ball $B_r^{2 n}$ in $(\R^{2 n},\omega_0)$.
Restricting $\varphi_k$ and $\varphi$ to $B_r^{2 n}$ and applying Corollary~\ref{cor:rig-symp-emb} yields $\varphi^* \omega = \omega$ at (and near) $x$.
Since the point $x$ was arbitrary, the proof is complete.
\end{proof}

\begin{cor} \label{cor:rig-anti-symp-diff}
Let $\varphi_k \colon W \to W$ be anti-symplectic diffeomorphisms that converge uniformly on compact subsets to a diffeomorphism $\varphi \colon W \to W$.
Then $\varphi$ is anti-symplectic, that is, $\varphi^* \omega = - \omega$.
In other words, the set of anti-symplectic diffeomorphisms is $C^0$-closed in the group of all diffeomorphism.
\end{cor}

\begin{proof}
This follows from Corollary~\ref{cor:rig-anti-symp-emb} and Darboux's Theorem (cf.\ the proof of Corollary~\ref{cor:rig-symp-diff}), or from Corollary~\ref{cor:rig-symp-diff} as in the proof of Corollary~\ref{cor:rig-anti-symp-emb} (if the set of anti-symplectic diffeomorphisms is empty, then there is nothing to prove.
\end{proof}

\begin{cor} \label{cor:rig-conf-symp-diff}
Let $\varphi_k \colon W \to W$ be a sequence of conformally symplectic diffeomorphisms that converges to a diffeomorphism $\varphi \colon W \to W$ uniformly on compact subsets.
Then $\varphi$ is conformally symplectic.
If $\varphi_k^* \omega = c_k \, \omega$, then the numbers $c_k$ converge to a non-zero constant $c$, and $\varphi^* \omega = c \, \omega$.
In particular, the group of conformally symplectic diffeomorphisms is $C^0$-closed in the group $\Diff (W)$.
The subgroup of diffeomorphisms for which $c > 0$ and the subset of diffeomorphisms for which $c < 0$ are also $C^0$-closed in the group $\Diff (W)$.
\end{cor}

\begin{proof}
This follows directly from Corollary~\ref{cor:rig-conf-symp-emb} and Darboux's Theorem by the same argument as in the proof of Corollary~\ref{cor:rig-symp-diff}.
\end{proof}

An embedding $\varphi \colon W_1 \to W_2$ is said to be shape rescaling if there exists a non-zero constant $c$ such that $I (U, L, \tau) \subset c \, I (V, L, \tau \circ \varphi^*)$ for all open subsets $U \subset W_1$ and $V \subset W_2$ such that $\overline{U} \subset W_1$ is compact and $\varphi (\overline{U}) \subset V$, and for every closed and connected $n$-dimensional manifold $L$ and homomorphism $\tau \colon H^1 (V, \R) \to H^1 (U, \R)$.
If $c = - 1$, we also say that $\varphi$ reverses shape.
It can be shown along the same lines as the proof of Proposition~\ref{pro:conf-symp-shape} that the number $c$ is unique.

The following two results are almost immediate corollaries of Theorem~\ref{thm:symp-shape-preserving} along the same lines as the proofs of Corollaries~\ref{cor:rig-anti-symp-emb} and \ref{cor:rig-conf-symp-emb}.
Again the first one is really a special case of the second one.

\begin{cor} \label{cor:anti-symp-shape-preserving}
An embedding $\varphi \colon B_r^{2 n} \to W$ is anti-symplectic if and only if it reverses the shape invariant.
\end{cor}

\begin{cor} \label{cor:conf-symp-shape-preserving}
An embedding $\varphi \colon B_r^{2 n} \to W$ is conformally symplectic if and only if it rescales the shape invariant.
Moreover, the rescaling constant coincides with the conformal factor of $\varphi$.
\end{cor}

\begin{proof}[Proofs]
That  anti-symplectic and conformally symplectic embeddings reverse and rescale shape, respectively, follows directly from the definition.
The converse is proved exactly as in the proofs of Corollaries~\ref{cor:rig-anti-symp-emb} and \ref{cor:rig-conf-symp-emb}, respectively.
A similar argument could also be applied directly in a proof that closely follows the line of argument in the proof of Theorem~\ref{thm:symp-shape-preserving}.
\end{proof}

\begin{rmk}
As further corollaries one can give alternate proofs of the corollaries in this section that concern anti-symplectic and conformally symplectic embeddings and diffeomorphisms.
The precise alternate proofs are too similar to our previous arguments to be duplicated. \qed
\end{rmk}

\begin{rmk}
The proof of Proposition~\ref{pro:continuous-symp-shape} applies almost verbatim to show that rescaling shape is a property that is preserved by uniform limits (on compact subsets) with $c = \liminf c_k$, though this also follows from the previous corollaries combined with Corollary~\ref{cor:rig-conf-symp-emb} (with $c = \lim c_k$). \qed
\end{rmk}

\section{(Maximal) coisotropic embeddings} \label{sec:coisotropic}
According to Lemma~\ref{lem:lagrangian}, every symplectic manifold contains enough Lagrangian submanifolds to distinguish conformally symplectic embeddings.
In particular, if an embedding $\varphi$ is not conformally symplectic, then there exists an embedded Lagrangian torus $T^n$ such that $\varphi (T^n)$ is not Lagrangian, and thus is immediately displaceable (Lemma~\ref{lem:non-Lagrangian} and Theorem~\ref{thm:displacement}).
Lemma~\ref{lem:imm-displ} below is the counterpart for coisotropic submanifolds of a contact manifold.

This section contains a few preliminary results, which may be known to an expert in contact topology.
For everyone else, the book \cite{geiges:ict08} is a good starting point.
We adopt the guiding principle to at least sketch proofs of any result that is not explicitly stated in \cite{geiges:ict08}.
Not the entire discussion of coisotropic submanifolds below is necessary for the proof of Theorem~\ref{thm:contact-shape-preserving}, but we choose to present a systematic treatment of coisotropic submanifolds since it requires little extra effort.

Let $(M, \xi)$ be a contact manifold, that is, $\xi \subset TM$ is a completely non-integrable codimension one tangent distribution.
That means that at least locally $\xi \subset TM$ can be written as the kernel of a one-form $\alpha$ so that $\alpha \wedge (d\alpha)^{n - 1} \not= 0$.
Unless it is explicitly mentioned otherwise, we assume that $\xi$ is coorientable, i.e.\ there exists a global one-form $\alpha$ as above (and in particular, $\alpha \wedge (d\alpha)^{n - 1}$ is nowhere vanishing on $M$).
We fix a coorientation of $\xi$, so that the contact form $\alpha$ is determined up to multiplication by a positive function.
Again unless explicitly stated otherwise, we assume that all contact embeddings preserve the given coorientations.

\begin{rmk}
To simplify notation, we often suppress the contact structure and contact form from the notation, and simply write $M$ for example for a contact manifold.
If a specific contact manifold $M$ admits a canonical contact structure or form that has previously been referenced, a statement about $M$ refers to these canonical choices.
Whenever there is a potential ambiguity however, the choices will be made explicit. \qed
\end{rmk}

\begin{rmk} \label{rmk:coorientation}
A sufficiently small open subset of a contact manifold $M$ is always coorientable, so that by restricting its domain and target if necessary, the domain and target of a given contact embedding are coorientable.
Thus we may assume that the contact structure $\xi$ in Theorem~\ref{thm:contact-shape-preserving} and in Corollary~\ref{cor:rig-contact-emb} is cooriented.
A given contact embedding automatically preserves coorientation for a consistent choice of coorientation on the domain and target.
Corollary~\ref{cor:rig-contact-emb} is to be interpreted in the sense that the limit $\varphi$ maps the hyperplane bundle $\xi_0$ to $\xi$, and preserves (or reverses) coorientation when a subsequence of the sequence $\varphi_k$ does.
A contact diffeomorphism $(M, \xi) \to (M, \xi)$ of course preserves one choice of coorientation if and only if it preserves the opposite choice (on both the domain and target). \qed
\end{rmk}

\begin{dfn} \label{dfn:coisotropic}
Let $(M, \xi = \ker \alpha)$ be a cooriented contact manifold of dimension $2 n - 1$, and $L$ as before be a closed and connected $n$-dimensional manifold.
An embedding $\iota \colon L \hookrightarrow M$ is called coisotropic (or pre-Lagrangian) if
\begin{enumerate}
\item $\iota$ (or its image $\iota (L)$) is transversal to the contact structure $\xi$, and
\item the distribution $\iota^* \xi = \ker (\iota^* \alpha) \subset TL$ can be defined by a closed one-form.
\end{enumerate}
A (closed and connected) $n$-dimensional submanifold $C$ of $(M, \xi)$ is called coisotropic if $TC$ is transversal to $\xi |_C$ and the codimension one distribution $\xi |_C \cap TC \subset TC$ can be defined by a closed one-form on $C$. \qed
\end{dfn}

\begin{rmk}
Condition (1) in the definition is equivalent to the assumption that the subspaces $(\iota^* \xi)_x \subset T_x L$ have constant codimension one.
Any one-form that defines the distribution $\iota^* \xi$ must be of the form $f \, \iota^* \alpha$ for a non-zero function $f$ on $L$.
The second condition is therefore equivalent to the existence of a contact form $\alpha' = g \alpha$, where $g > 0$, so that $\iota^* \alpha' = (g \circ \iota) \iota^* \alpha$ is a closed one-form (the function $f \circ \iota^{-1}$, which is defined on the image of $\iota$, can be extended to a tubular neighborhood of $\iota (L)$, and then to a globally defined everywhere non-zero function $g$ on $M$; after replacing $g$ by $- g$ if necessary, we may assume that $g$ is positive).
As a consequence of condition (2), the distribution $\iota^* \xi \subset TL$ is integrable (and $\iota (L)$ is foliated by Legendrian submanifolds).
By the contact condition, transversality of $\iota$ and $\xi$ is a necessary condition for the latter.

Similarly, the definition of coisotropic submanifold is equivalent to the existence of a smooth positive function $g$ such that the restriction of $g \alpha$ to $C$ is closed.
Of course the image of a coisotropic embedding is a coisotropic submanifold. \qed
\end{rmk}

\begin{rmk}
The definition of coisotropic makes sense for submanifolds of larger codimension.
However, all coisotropic submanifolds in this paper are assumed to be closed, connected, and of (maximal) dimension $n$. \qed
\end{rmk}

Recall that the restriction of $d\alpha$ to the hyperplane bundle $\xi$ is a symplectic bundle structure, which does not depend on the choice of contact form $\alpha$ up to conformal rescaling.
In particular, the symplectic orthogonal complement of a subspace of $\xi$ is independent of the choice of contact form.

\begin{lem} \label{lem:description-coisotropic}
Let $C$ be a (closed and connected) $n$-dimensional submanifold of $(M, \xi)$ that is transversal to $\xi |_C$.
Then $C$ is coisotropic if and only if $\xi |_C \cap TC$ is a Lagrangian subbundle of $\xi |_C$ and $TC = (\xi |_C \cap TC) \oplus \langle R \rangle$, where $R$ is the Reeb vector field of a contact form $\alpha$ that defines $\xi$.
In fact, a closed one-form on $C$ that defines the distribution $\xi |_C \cap TC$ is given by the restriction of $\alpha$ to $C$.
\end{lem}

\begin{proof}
Suppose that $C$ is coisotropic, and let $\alpha$ be a contact form on $(M, \xi)$ whose restriction to $C$ is closed.
In particular, $d\alpha$ restricted to $\xi |_C \cap TC$ vanishes, and therefore the latter is Lagrangian.
Denote by $R$ the Reeb vector field of $\alpha$.
Let $x \in C$ and $v \in \xi_x$ be a vector so that $R (x) + v \in T_x C$.
Since $d\alpha$ is zero when restricted to $TC$, $v$ belongs to the Lagrangian complement of $\xi_x \cap T_x C$.
We have already shown the latter to be a Lagrangian subspace of $\xi_x$, and thus $v \in \xi_x \cap T_x C$.
Thus $TC = (\xi |_C \cap TC) \oplus \langle R \rangle$ as claimed.
In the converse direction, the one-form $d\alpha |_C$ is clearly closed, and therefore $C$ is coisotropic.
\end{proof}

\begin{lem} \label{lem:existence-coisotropic-tori}
Let $(M, \xi)$ be a contact manifold, $x \in M$ be a point, $v \notin \xi_x$ be a vector, $W \subset \xi_x$ be an isotropic subspace, and $U \subset M$ be a neighborhood of $x$.
Then there exists a coisotropic embedding $T^n \hookrightarrow M$ through the point $x$, whose image $C$ is contained in $U$.
Moreover, we may assume that $v$ and $W$ are tangent to $C$, and in fact, that $v = R (x)$ for a Reeb vector field as in Lemma~\ref{lem:description-coisotropic}.
\end{lem}

\begin{proof}
It is an elementary fact in contact topology that any open subset $U \subset M$ contains an embedded transverse knot, i.e.\ an embedding $S^1 \hookrightarrow M$ that is transverse to $\xi$.
This can easily be seen from the fact that $\xi$ is a hyperplane bundle that is nowhere integrable, see a picture of the standard contact structure on $\R^3$ (which canonically embeds into $(\R^{2 n - 1}, \xi_0)$) for example on page 4 of \cite{geiges:ict08}.
By Darboux's Theorem, we can identify a neighborhood of the origin with a neighborhood of $x$ in $M$ so that $x$ corresponds to the origin and the vector $v$ to $\partial / \partial z$.
In particular, we may choose the transverse knot tangent to $v$.
(We would like to alert the reader that the contact structures defined by all of the following contact forms are referred to as the standard contact structure on $\R^{2 n - 1}$ in various places in the literature: $dz - \sum_{i = 1}^{n - 1} y_i \, dx_i$, $dz + \sum_{i = 1}^{n - 1} x_i \, dy_i$, $dz + \sum_{i = 1}^{n - 1} (x_i \, dy_i - y_i \, dx_i)$, and $dz + \frac{1}{2} \sum_{i = 1}^{n - 1} (x_i \, dy_i - y_i \, dx_i)$.
However, all of these contact forms are easily seen to be mutually diffeomorphic by writing down explicit diffeomorphisms.)

By the Contact Neighborhood Theorem, a sufficiently small neighborhood of this knot (inside the set $U$) is contact diffeomorphic to an open neighborhood $S^1 \times V$ of $S^1 \times 0 \subset S^1 \times \R^{2 n - 2}$ with the contact structure (still denoted by $\xi$) induced by the contact form
	\[ \alpha = dz + \frac{1}{2} \sum_{i = 1}^{n -1} (x_i \, dy_i - y_i \, dx_i) = dz + \frac{1}{2} \sum_{i = 1}^{n -1} r_i^2 \, d\theta_i, \]
where $z \in S^1$, and $x_i = r_i \cos \theta_i$ and $y_i = r_i \sin \theta_i$ are coordinates on $\R^{2 n - 2}$, see Theorem~2.5.15 and Example~2.5.16 in \cite{geiges:ict08}.
Let $L \subset \R^{2 n - 2}$ be an (embedded) Lagrangian torus with respect to the standard symplectic structure $d\alpha |_\xi = \omega_0$.
Then the $n$-torus $S^1 \times L$ is an (embedded) coisotropic submanifold of $M$.
(This notation is not meant to suggest that $L$ is everywhere tangent to $\xi$, which is of course impossible by the contact condition, i.e.\ the twisting of the hyperplane bundle $\xi$.
For instance, if $L$ is a split torus where $r_i = c_i > 0$ is constant for all $i$, then $\xi_{(z, r, \theta)}$ is spanned by the vectors $\partial / \partial r_i$ and $\partial / \partial \theta_i - \frac{1}{2} c_i^2 \, \partial / \partial z$ for $i = 1, \ldots, n - 1$.)
If $L$ contains the origin, then the image of the embedding contains $x$, and the vector $v = \partial / \partial z = R (x)$.
Finally, choose $L$ so that $W$ is tangent to $C$ (cf.\ the proof of Lemma~\ref{lem:lagrangian}).
\end{proof}

\begin{rmk} \label{rmk:small-coiso}
For later reference, we point out that if the numbers $r_i > 0$, $i = 1, \dots, n - 1$ are sufficiently small, then the image of the standard (up to translation) Lagrangian embedding $\jmath$ of the split torus $S^1 (r_1) \times \cdots \times S^1 (r_{n -1})$ into $\R^{2 n - 2}$ is contained in $V$; the embedding $\iota \colon T^n \hookrightarrow S^1 \times V$, $\iota (z, x) = (z, \jmath (x))$ is coisotropic, the one-form $\iota^* \alpha$ is closed, and $[\iota^* \alpha] = (2 \pi, \pi r_1^2, \ldots, \pi r_{n -1}^2) \in \R^n = H^1 (T^n, \R)$.
Here we identify the first $S^1$-factor with $\R / (2 \pi \Z)$.
In general, the first coordinate depends on the size of the transverse knot, and can be chosen to be any sufficiently small non-zero number.
If $V = \R^{2 n - 2}$, this holds for all positive numbers $r_i$. \qed
\end{rmk}

The following theorem is the analog of Weinstein's Lagrangian Neighborhood Theorem for coisotropic submanifolds.
(On the other hand, the symplectic version of Lemma~\ref{lem:existence-coisotropic-tori} is an immediate consequence of Darboux's Theorem.)

\begin{thm} \label{thm:neighborhood}
Let $\iota \colon L \hookrightarrow (M, \xi)$ be a coisotropic embedding, and $\alpha$ be a contact form on $(M, \xi)$ so that $\iota^* \alpha$ is closed.
Let $\beta = p \circ \iota^* \alpha \colon L \hookrightarrow ST^*L$ be the section defined by $\iota^* \alpha$, where $p \colon T^*L \backslash L_0 \to ST^*L$ is the obvious map, and $L_0$ denotes the zero section.
Then there exists a neighborhood $U$ of $\beta$ in $ST^*L$, a neighborhood $V$ of $\iota (L)$ in $M$, and a contact diffeomorphism $\varphi \colon U \to V$ that restricts to the identity on $L$, that is, $\iota = \varphi \circ \beta$.
\end{thm}

\begin{rmk}
Recall that in Weinstein's Lagrangian Neighborhood Theorem, a neighborhood of (the image of) a Lagrangian embedding $\iota \colon L \hookrightarrow W$ of a compact manifold is identified with a neighborhood of the zero section in $T^*L$.
To see why $\iota (L)$ can be identified with the zero section (and not only a specific section of $T^*L$), note that for a sufficiently small tubular neighborhood $U$ of $\iota (L)$ in $M$, the inclusion $\iota \colon L \hookrightarrow U$ induces an isomorphism $H^1 (U, \R) \to H^1 (L, \R)$, and thus there exists a choice of primitive one-form $\lambda$ of the symplectic form on $U$ that makes $\iota^* \lambda$ exact, see Remark~\ref{rmk:exactness}.

On the other hand, there exist contact invariants of coisotropic submanifolds that are not present for (or correspond to any symplectic invariants of) Lagrangian submanifolds.
In dimension $3$, the characteristic foliation (see section~\ref{sec:contact} for the definition) is one example, but there are others (in all dimensions).
Possibly the simplest one occurs when $C = T^n \times a \subset T^n \times S^{n - 1}$ (the unit cotangent bundle of a torus with its standard contact structure), where $a = (a_1, \ldots, a_n) \in S^{n - 1}$ is a point.
As already pointed out above, $T^n \times a$ is a coisotropic submanifold, and $\iota_a^* \alpha_\can$ is a closed one-form on $T^n$.
Suppose that $f$ is a (positive) smooth function on $T^n$ so that $f \, \iota_a^* \alpha_\can$ is also closed.
Then $df (v) = 0$ for every vector $v$ that is tangent to $\xi |_C \cap TC$, and $f$ must be of the form $f (q) = g (\langle a, q \rangle)$ for a function $g$ of a single variable, where $q = (q_1, \dots, q_n) \in T^n$.
If $a$ is irrational, then $f$ must be constant, and thus in this case the function $f$ in Definition~\ref{dfn:tilde-contact-shape} is unique (up to rescaling by a positive constant).
In particular, the Reeb foliation on $T^n \times a$ is unique and provides a contact invariant of the coisotropic submanifold.
If $a$ is rational, $g$ can be identified with a smooth function on $S^1$ (and the space of functions $f$ as in Definition~\ref{dfn:tilde-contact-shape} with $C^\infty (S^1, \R)$).
The cohomology class $[f \, \iota_a^* \alpha_\can]$ lies on the oriented line through $[\iota_a^* \alpha_\can] \in H^1 (T^n, \R)$, and thus defines the same element in $PH^1 (T^n, \R)$. \qed
\end{rmk}

\begin{proof}[Proof of Theorem~\ref{thm:neighborhood}]
It is possible to give a proof of the theorem using Weinstein's Lagrangian Neighborhood Theorem on the symplectization.
We prefer to give a purely contact geometric proof here.
The two proofs are to a large extend dual to one another.
For convenience, we identify both $\iota (L)$ and the section $\beta$ with $L$.

Let $g$ be a Riemannian metric on $L$ and $J$ be an almost complex structure on $\xi$ that are compatible with $\alpha$ in the sense that $g = \alpha \otimes \alpha + d\alpha |_\xi (\cdot, J \cdot)$ \cite{blair:rgc10}.
We assume that the unit cotangent bundle $ST^*L$ is determined by this metric.
As a contact manifold, the latter does not depend on the choice of Riemannian metric.
In fact, we may identify $ST^*L$ (as a contact manifold with its canonical contact structure) with the oriented projectivization $PT^*L = (T^*L \backslash L_0) / \R_+$, where $L_0$ again denotes the zero section.
That is, an element of $ST^*L$ can be considered as an oriented line in $T^*L$.
A choice of Riemannian metric on $L$ is required however to define the canonical contact form $\alpha_\can = \lambda_\can |_{ST^*L}$, and any two choices yield naturally diffeomorphic contact manifolds.

Since $L$ is coisotropic, the normal bundle of $L$ in $TM$ is given by $NL = J (\xi |_L)$.
Thus the tangent bundle $TL$ is isomorphic to the direct sum bundle $\R \oplus NL$, and the Riemannian metric induces an isomorphism of the latter with $T^*L$.
(In the symplectization, the factor $\R$ is generated by the one-form $dt$.)
In a neighborhood of $\iota^* \alpha$, the $\R$-component is always non-zero (an oriented line in $ST^*L$ is not orthogonal to the $\R$-factor), and thus we may identify $ST^*L$ with $NL = 1 \oplus NL$.

The Riemannian metric induces an isomorphism $T (ST^*L) = TL \oplus ST^*L$, see Exercises~3.10 and 3.11 in \cite{mcduff:ist98}.
The canonical contact form $\alpha_\can$ restricted to the section $\beta$ is by definition $\beta \circ d\pi$, where $\pi \colon ST^*L \to L$ is the canonical projection, and thus $\alpha$ and $\alpha_\can$ coincide on $L$.
By construction, the two-forms $d\alpha$ and $d\alpha_\can$ also agree on $L$: the above isomorphism identifies the restriction of the latter to $\xi |_L$ with the canonical two-form on $\xi |_L \oplus ST^*L \subset TL \oplus T^*L$ (again see Exercise~3.10 in \cite{mcduff:ist98}), which in turn is the two-form $d\alpha |_\xi = g |_\xi (J \cdot, \cdot)$ on $\xi |_L \oplus NL$ under the previous identifications.
The conclusion then follows from a Gray stability argument verbatim as in the first paragraph of the proof of Theorem~2.5.15 in \cite{geiges:ict08}.
\end{proof}

Combining the previous results yields the following proposition.

\begin{pro}[\cite{ms:gae14}] \label{pro:coisotropic-darboux-weinstein}
Let $(M, \xi)$ be a contact manifold, $U \subset M$ be an open subset, and $a = (a_1, \ldots, a_n) \in S^{n -1}$ be a point with $0 < a_i < 1$ for all $i$.
If $a_i$, $i = 2, \ldots, n$, are sufficiently small, then there exists a neighborhood $A$ of $a$ in $S^{n - 1}$ and a contact embedding of $T^n \times A$ with its standard contact structure into $U$.
If $x \in M$, we may in addition assume that $x$ lies in the image of $T^n \times a$.
\end{pro}

\begin{rmk}
Using the action of $GL (n, \Z)$ on $T^n$, the restrictions on the point $a$ can be relaxed to some extend (see Proposition~\ref{pro:P-L-diffeos} below).
The proof in \cite{ms:gae14} gives an explicit construction of the contact diffeomorphism between a neighborhood of $T^n \times a$ and an open subset of $S^1 \times \R^{2 n - 2}$ with the canonical contact structure induced by the contact form $dz + \frac{1}{2} \sum_{i = 1}^{n -1} r_i^2 d\theta_i$.
Let $\O$ be the open subset of $S^1 \times \R^{2 n - 2}$ on which $r_i > 0$ for all $i$, and define a function $r \colon \O \to (0,1)$ by $r = (1 + \frac{1}{4} \sum_{i = 1}^{n - 1} r_i^4)^{- \frac{1}{4}}$.
Then the map $\O \to T^n \times S^{n - 1}$ defined by
	\[ (z, r_1, \ldots, r_{n - 1}, \theta_1, \ldots, \theta_{n -1}) \mapsto \left( z, \theta_1, \ldots, \theta_{n - 1}, r^2, \frac{1}{2} (r \cdot r_1)^2, \ldots, \frac{1}{2} (r \cdot r_{n - 1})^2 \right) \]
is a contact diffeomorphism onto the subset $\P$ of $T^n \times S^{n -1}$ on which all spherical coordinates are positive.
The intersection $(S^1 \times V) \cap \O$ is open and contains the (preimage of the) point $a$ provided that $a_i > 0$ are sufficiently small for all $i > 1$.

The coisotropic submanifold $S^1 \times T^{n - 1}$ constructed in the course of the proof of Lemma~\ref{lem:existence-coisotropic-tori}, where $T^{n - 1}$ is a split torus, is mapped to the coisotropic submanifold $T^n \times a \subset T^n \times S^{n - 1}$ by this diffeomorphism.
Its shape is $[\iota^* \alpha_\can] = a$, which is consistent with the computation in Remark~\ref{rmk:small-coiso} (the two vectors differ by the constant factor $2 \pi / r^2$ and thus belong to the same oriented line). \qed
\end{rmk}

\section{Eliashberg's contact shape invariants} \label{sec:contact-shape}
We recall the definitions of two shape invariants for contact manifolds from \cite{eliashberg:nio91}.
The second version is a refinement of the first one; both are contact invariants.

Denote by $M \times \R_+$ the symplectization of $(M,\alpha)$ endowed with the symplectic structure $\omega = d\lambda$, where $\lambda = t \, \pi^* \alpha$ (the Liouville one-form), $\pi \colon M \times \R_+ \to M$ is the projection to the first factor, and $t$ is the coordinate on the factor $\R_+ = (0,\infty)$.
Up to an exact symplectic diffeomorphism (i.e.\ that interchanges the corresponding Liouville one-forms), the symplectization depends only on the contact structure $\xi$, and not on the particular choice of contact form $\alpha$ with $\ker \alpha = \xi$.
A contact embedding $\varphi \colon (M_1,\xi_1 = \ker \alpha_1) \to (M_2,\xi_2 = \ker \alpha_2)$ with $\varphi^* \alpha_2 = f \alpha_1$, where $f$ is a positive function on $M$, induces an ($\R_+$-equivariant) symplectic embedding $(M_1 \times \R_+, d(t \, \pi^* \alpha_1)) \to (M_2 \times \R_+, d(t \, \pi^* \alpha_2))$ given by $(x, t) \mapsto (\varphi (x), t / f (x))$.
That is, the lift of a contact embedding preserves not only the symplectic structure but also the Liouville one-form $\lambda$ itself.
Since the product $M \times \R_+$ deformation retracts onto the first factor $M$, a given homomorphism $\tau \colon H^1 (M,\R) \to H^1 (L, \R)$ can be identified with a homomorphism $H^1 (M \times \R_+, \R) \to H^1 (L, \R)$.
Thus the symplectic shape $I (M \times \R_+, \lambda, L, \tau)$ of the symplectization is a contact invariant of the contact manifold $(M, \xi)$, and in contrast to the symplectic case, we may define the shape as a contact invariant of the Liouville one-form $\lambda$ without any freedom of translation.

\begin{rmk}
The map $\sigma_s \colon (x, t) \mapsto (x, s \, t)$, $s > 0$, is an $\R_+$-equivariant conformal symplectic diffeomorphism that is isotopic to the identity.
It thus follows from the definition that $I (M \times \R_+, \lambda, L, \tau)$ is a cone in $H^1 (L, \R)$.
It does not contain its vertex.
That last fact is well known, but maybe the proof deserves to be repeated here.
Suppose that $\iota$ is an exact Lagrangian embedding $L \hookrightarrow M \times \R_+$.
Then the Lagrangian embedding $\iota_s = \sigma_s \circ \iota$ is also exact, and therefore there exists a Hamiltonian isotopy $\psi_t \colon M \times \R_+ \to M \times \R_+$ such that $\psi_0$ is the identity and $\iota_s = \psi_s \circ \iota$, see e.g.\ Exercise~11.26 in \cite{mcduff:ist98}.
For $s$ sufficiently large, the image of $\iota_s$ does not intersect the image of $\iota$, which contradicts Gromov's Theorem, cf.\ Remark~11.21 in \cite{mcduff:ist98}. \qed
\end{rmk}

Thus $I (M \times \R_+, \lambda, L, \tau)$ is a cone without its vertex in $H^1 (L, \R)$.
It is therefore convenient to projectivize the invariant (in the oriented sense of identifying vectors that differ by a positive scalar factor).

\begin{dfn} \label{dfn:contact-shape}
The contact $(L, \tau)$-shape of $(M, \xi)$ is the subset
	\[ I_C (M, L, \tau) = I_C (M, \xi, L, \tau) = PI (M \times \R_+, \lambda, L, \tau) = I (M \times \R_+, \lambda, L, \tau) / \R_+ \]
of $PH^1 (L, \R) = H^1 (L, \R) / \R_+$, that is, the projectivization of the set of all $z$ in $H^1 (L, \R)$ such that there exists a Lagrangian embedding $\iota \colon L \hookrightarrow M \times \R_+$ with $\iota^* = \tau$ and $[\iota^* (t \, \pi^* \alpha)] = z$. \qed
\end{dfn}

The following result is the analog of Proposition~\ref{pro:symp-shape} for contact embeddings.
We give a proof here to illustrate why the (symplectic) shape invariant is more suitable to study contact embeddings than other symplectic invariants (such as symplectic capacities) of the symplectization of a contact manifold.

\begin{pro}[{\cite{eliashberg:nio91}}] \label{pro:contact-shape}
Let $(M_1, \xi_1 = \ker \alpha_1)$ and $(M_2, \xi_2 = \ker \alpha_2)$ be two contact manifolds of the same dimension, and let $\varphi \colon M_1 \to M_2$ be a contact embedding.
Then $I_C (M_1, L, \tau) \subset I_C (M_2, L, \tau \circ \varphi^*)$.
If $\varphi$ is a contact diffeomorphism, then $I_C (M_1, L, \tau) = I_C (M_2, L, \tau \circ \varphi^*)$.
\end{pro}

\begin{proof}
Let $\wv (x, t) = (\varphi (x), t / f (x))$ be the lift of the contact embedding $\varphi$ to the symplectization; it maps the cone $M_1 \times \R_+$ into the cone $\varphi (M_1) \times \R_+ \subset M_2 \times \R_+$ (in fact, $\wv (M_1 \times \R_+) = \varphi (M_1) \times \R_+$).
Recall that we identify ${\wv}^*$ with $\varphi^*$, and likewise for the homomorphism $\tau$.
Since $\wv$ is a symplectic embedding (that preserves the one-form $\lambda$), the inclusion $I (M_1 \times \R_+, \lambda, L, \tau) \subset I (M_2 \times \R_+, \lambda, L, \tau \circ \varphi^*)$ holds by Proposition~\ref{pro:symp-shape}.
The claim now follows from the definition of the invariant $I_C$.
\end{proof}

We next recall the definition of the modified contact shape $\It_C$ from \cite{eliashberg:nio91}.
An advantage of this invariant is that it is defined in intrinsic contact terms without the use of the symplectization.
Thus for the definition of this invariant (and in fact for coisotropic embeddings as well), it is not necessary to assume that the contact structure is cooriented.
The latter is necessary however to compare the modified shape to the shape invariant $I_C$, and for convenience, we give the definition for cooriented contact structures only.

\begin{dfn} \label{dfn:tilde-contact-shape}
The modified contact $(L, \tau)$-shape $\It_C (M, L, \tau) = \It_C (M, \xi, L, \tau) \subset PH^1 (L, \R)$ of $(M, \xi)$ is by definition the projectivization of the set of all points $z \in H^1 (L, \R)$ such that there exists a coisotropic embedding $\iota \colon L \hookrightarrow M$ and a positive function $f$ on $L$, so that $\iota^* = \tau$, the one-form $f \, \iota^* \alpha$ is closed and defines the codimension one distribution $\iota^* \xi$, and $z = [f \, \iota^* \alpha]$. \qed
\end{dfn}

\begin{rmk} \label{rmk:shape-no-coorientation}
Note that if $\beta$ is closed and defines the distribution $\iota^* \xi$, then so does $s \beta$ for any $s \not= 0$.
Moreover, since a smooth function on a closed manifold must have a critical point, a codimension one distribution cannot be defined by an exact one-form.
Thus the subset of $H^1 (L, \R)$ that appears in the definition of the modified contact shape is again a cone without its vertex. \qed
\end{rmk}

\begin{rmk}
The other choice of coorientation of $\xi$ replaces the cone $\It_C (M, L, \tau)$ by its opposite $- \It_C (M, L, \tau)$.
If one chooses to ignore coorientation, the modified contact shape can be defined as the union of these two cones, and similarly for the original contact shape invariant in Definition~\ref{dfn:contact-shape}. \qed
\end{rmk}

\begin{pro} \label{pro:tilde-contact-shape}
Let $(M_1, \xi_1 = \ker \alpha_1)$ and $(M_2, \xi_2 = \ker \alpha_2)$ be two contact manifolds of the same dimension, and let $\varphi \colon M_1 \to M_2$ be a contact embedding.
Then $\It_C (M_1, L, \tau) \subset \It_C (M_2, L, \tau \circ \varphi^*)$.
If $\varphi$ is a contact diffeomorphism, then $\It_C (M_1, L, \tau) = \It_C (M_2, L, \tau \circ \varphi^*)$.
\end{pro}

\begin{proof}
Consider the positive function $g$ on $M_1$ defined by the relation $\varphi^* \alpha_2 = g \alpha_1$.
Let $\iota \colon L \hookrightarrow M_1$ be a coisotropic embedding, and $\beta = f \, \iota^* \alpha_1$ be a closed one-form that defines the distribution $\iota^* \xi_1$.
Then the embedding $\varphi \circ \iota \colon L \hookrightarrow M_2$ is coisotropic, and $(f / (g \circ \iota)) (\varphi \circ \iota)^* \alpha_2 = \beta$.
\end{proof}

Again we are mostly interested in the situation in which $L = T^n$ and $M$ is an open subset of either $\R^{2 n - 1}$ with its standard contact structure or of the unit cotangent bundle $ST^* T^n \subset T^* T^n$ of $T^n$ (with respect to some Riemannian metric) with its canonical contact structure $\xi_\can = \ker \alpha_\can$, where $\alpha_\can$ is the restriction of the canonical one-form $\lambda_\can$ on $T^* T^n$.
The symplectization of $ST^* T^n$ is diffeomorphic to $T^* T^n$ minus the zero section with its standard symplectic structure $\omega_\can$ via the diffeomorphism $(q, p, t) \mapsto (q, t \, p)$.
The trivialization $T^* T^n = T^n \times \R^n$ restricts to the trivialization $ST^* T^n = T^n \times S^{n - 1}$.
This gives rise to an identification of the (oriented) projectivized group $PH^1 (T^n, \R)$ with the fiber $S^{n -1}$ of the fibration $ST^* T^n = T^n \times S^{n - 1} \to T^n$.
As in the symplectic case, for brevity we often omit the manifold $T^n$ from the notation.

\begin{dfn}
Define $I_C (M, \tau) = I_C (M, T^n, \tau)$ and $\It_C (M, \tau) = \It_C (M, T^n, \tau)$.
\end{dfn}

For a given homomorphism $\Phi \colon H^1 (L, \R) \to H^1 (L, \R)$, denote by $P \Phi$ the induced homomorphism $PH^1 (L, \R) \to PH^1 (L, \R)$.
The analog of Proposition~\ref{pro:L-diffeos} holds for the contact shapes; the proof is almost verbatim the same and thus is omitted.

\begin{pro} \label{pro:P-L-diffeos}
If $\phi \colon L \to L$ is a diffeomorphism, then the two shapes satisfy $I_C (M, L, \phi^* \circ \tau) = P \phi^* (I_C (M, L, \tau))$ and $\It_C (M, L, \phi^* \circ \tau) = P \phi^* (\It_C (M, L, \tau))$.
In particular, $I_C (M, A \circ \tau) = P A (I_C (M, \tau))$ and $\It_C (M, A \circ \tau) = P A (\It_C (M, \tau))$ provided that $A \in GL (n, \Z)$.
\end{pro}

The modified contact shape is related to the original contact shape by means of the following proposition and its corollary.

\begin{pro} \label{pro:lift-embedding}
Let $\iota \colon L \hookrightarrow M$ be an embedding, and let $\hi_f \colon L \hookrightarrow M \times \R_+$ denote the embedding $\hi_f (x) = (\iota (x), f (x))$ into the symplectization of $(M, \alpha)$, where $f$ is a positive function on $L$.
Then the embedding $\iota$ is coisotropic if and only if the embedding $\hi_f$ is Lagrangian for some $f > 0$.
Equivalently, there exists a contact form $\alpha'$ on $(M, \xi)$ so that the embedding $\hi_1 (x) = (\iota (x), 1)$ into the symplectization of $(M, \alpha')$ is Lagrangian.
In fact, the one-form $f \, \iota^* \alpha$ (that defines the distribution $\iota^* \xi$) is closed if and only if $\hi_f$ is Lagrangian, and $\alpha' = f \alpha$.
\end{pro}

\begin{proof}
It is immediate to verify that $\hi_f^* (t \, \pi^* \alpha) = f \, \iota^* \alpha$ and the diffeomorphism $\sigma_f (x, t) = (x, t / f (x))$ of $M \times \R_+$ satisfies $\sigma_f^* (t \, \pi^* (f \alpha)) = t \, \pi^* \alpha$.
\end{proof}

\begin{cor}[{\cite{eliashberg:nio91}}] \label{cor:contact-shape-rel}
$\It_C (M, L, \tau) \subset I_C (M, L, \tau)$ for every (cooriented) contact manifold $(M, \xi = \ker \alpha)$ and every $L$ and $\tau$ as in the definitions of the shapes.
\end{cor}

In certain situations, the shapes can be (partly) calculated explicitly, as the following lemmas and proposition show.

\begin{lem}
If $(M, \xi)$ is any contact manifold, and the homomorphism $\tau$ factors through a composition $H^1 (M, \R) \to H^1 (S^1, \R) \to H^1 (T^n, \R)$, where $S^1 \subset M$ is an embedded circle, and the first map is induced by its inclusion, then $I_C (M, \tau)$ and $\It_C (M, \tau)$ are non-empty.
\end{lem}

\begin{proof}
By Theorem~3.3.1 in \cite{geiges:ict08}, every circle embedding can be $C^0$-approximated by a transverse knot that is isotopic to the original embedding.
For the modified shape the lemma thus follows from the construction in the proof of Lemma~\ref{lem:existence-coisotropic-tori}, and for the original shape it then follows from Corollary~\ref{cor:contact-shape-rel}.
\end{proof}

\begin{lem}
$\It_C (M, 0) = I_C (M, 0) = S^{n - 1}$ for any contact manifold $(M, \xi)$.
\end{lem}

\begin{proof}
By Proposition~\ref{pro:tilde-contact-shape} and Corollary~\ref{cor:contact-shape-rel} (and Darboux's Theorem), it suffices to show that $S^{n - 1} \subset \It_C (\R^{2 n - 1}, 0)$.
(Here we also use the standard contact dilation $(z, x, y) \mapsto (s^2 z, s x, s y)$ of $\R^{2 n - 1}$.)
By the same argument as in Remark~\ref{rmk:small-coiso}, the element $(1 : a_1 : \ldots : a_n) \in \It_C (\R^{2 n - 1}, 0)$ for all $a_i > 0$.
Then by (the $GL (n, \Z)$-action in) Proposition~\ref{pro:P-L-diffeos}, we have $S^{n - 1} \subset \It_C (\R^{2 n - 1}, 0)$.
\end{proof}

\begin{lem}
Consider $S^1 \times \R^{2 n - 2}$ with its standard contact structure, and let $\tau \colon H^1 (S^1 \times U, \R) \to H^1 (T^n, \R)$ be the homomorphism induced by the canonical embedding $S^1 \times T^{n - 1} \to S^1 \times 0 \hookrightarrow S^1 \times \R^{2 n - 2}$, where $U$ is an open subset of $\R^{2 n - 2}$.
Then the shape $\It_C (S^1 \times U, \tau)$ equals the upper hemisphere of $S^{n - 1}$ minus the north pole.
\end{lem}

\begin{proof}
By Proposition~1.24 in \cite{eliashberg:gct06}, for any positive numbers $r_1$ and $r_2$ there exists a contact embedding $\varphi \colon S^1 \times B_{r_1}^{2 n - 2} \to S^1 \times B_{r_2}^{2 n - 2}$ that is trivial on the first factor, and thus $\tau \circ \varphi^* = \tau$.
(In fact, for $n > 1$, the induced homomorphism $\varphi^*$ is the identity.)
Thus it suffices to prove the lemma for the case $U = \R^{2 n - 1}$.
Again by Remark~\ref{rmk:small-coiso}, we have $(1 : a_2 : \ldots : a_n) \in \It_C (S^1 \times \R^{2 n - 2}, \tau)$ for all $a_i > 0$, and by Proposition~\ref{pro:P-L-diffeos}, the shape $\It_C (S^1 \times \R^{2 n - 2}, \tau)$ contains the upper hemisphere minus the north pole.

On the other hand, since the canonical map $S^1 = S^1 \times T^{n - 1} \hookrightarrow S^1 \times \R^{2 n - 2} \to S^1$ has degree one for any (coisotropic) embedding $\iota \colon S^1 \times T^{n - 1} \hookrightarrow S^1 \times \R^{2 n - 2}$, the shape $\It_C (S^1 \times \R^{2 n - 2}, \tau)$ does not contain any point $(a_1 : a_2 : \ldots : a_n)$ with $a_1 \le 0$.
Moreover, it does not contain the north pole $(1 : 0 : \ldots : 0)$ \cite[Theorem~1.4]{mueller:hdp17}.
\end{proof}

As another corollary to Theorem~\ref{thm:neighborhood}, we have the following lemma.
The proof is analogous to the proof of Lemma~\ref{lem:open}, and also follows from Proposition~\ref{pro:tilde-contact-shape} by the same argument as in the symplectic case.

\begin{lem} \label{lem:contact-shape-open}
$I_C (W, L, \tau)$ and $\It_C (W, L, \tau)$ are open subsets of $PH^1 (L, \R)$.
\end{lem}

If $M = T^n \times A \subset T^n \times S^{n - 1}$ is an open and connected subset of the unit cotangent bundle of a torus with its standard contact structure, then the contact shapes can again be calculated completely.

\begin{pro} \label{pro:shape-unit-cotangent-bundle}
If $A \subset S^{n - 1}$ is an open and connected subset, and $a \in A$ is any point, then $\It_C (T^n \times A, \iota_a^*) = I_C (T^n \times A, \iota_a^*) = A$.
\end{pro}

\begin{proof}
As above denote by $\alpha_\can$ the canonical contact form on $ST^* T^n$, that is, the restriction of the canonical one-form $\lambda_\can$ on $T^* T^n$ to $ST^* T^n$.
Its Reeb vector field is $R = \sum_{i = 1}^n p_i \cdot \partial / \partial q_i$.
Thus for a point $a' \in A \subset S^{n - 1}$, the canonical embedding $\iota_{a'} \colon T^n = T^n \times a' \hookrightarrow ST^* T^n$ is transversal.
The one-form $\iota_{a'}^* \alpha_\can$ is closed, and $[\iota_{a'}^* \alpha_\can] = a'$.
Thus $A \subset \It_C (T^n \times A, \iota_a^*)$.
On the other hand \cite{eliashberg:nio91}, let $CA = \{ t a' \mid a' \in A, t \in \R_+ \}$ denote the cone over $A$ without its vertex.
Then $I (T^n \times CA, \lambda_\can, \iota_a^*) = CA$ by Theorem~\ref{thm:torus-shape}, and hence $I_C (T^n \times A, \iota_a^*) = A$ by definition of the contact shape $I_C$.
\end{proof}

\begin{rmk} \label{rmk:capacities-infinite}
In contrast, the capacity of the symplectization is infinite for any contact manifold.
Indeed, let $c (W, \omega)$ denote the Gromov width of a symplectic manifold $(W, \omega)$.
Then $0 < \pi r^2 < c (M \times \R_+, d (t \, \pi^* \alpha))$ by Darboux's Theorem.
On the other hand, the existence of the diffeomorphism $(x, t) \mapsto (x, s \, t)$ implies that $c (M \times \R_+, d (t \, \pi^* \alpha)) = c (M \times \R_+, s \, d (t \, \pi^* \alpha))$ for every $s > 0$ by the monotonicity axiom, and $c (M \times \R_+, s \, d (t \, \pi^* \alpha)) = s^2 \, c (M \times \R_+, d (t \, \pi^* \alpha))$ by the conformality axiom.
Thus $c (M \times \R_+, d (t \, \pi^* \alpha)) > \pi (r / s)^2 \to + \infty$ (as $s \to 0^+$), which proves that $c (M \times \R_+, d (t \, \pi^* \alpha)) = + \infty$.
Since the Gromov width is the smallest capacity, the capacity of the symplectization of any contact manifold is always infinite. \qed
\end{rmk}

The following definition is the contact analog of Definition~\ref{dfn:symp-shape-preserving}.
It likewise plays a crucial role in the proof of Corollary~\ref{cor:rig-contact-emb}.

\begin{dfn} \label{dfn:contact-shape-preserving}
Let $(M_1, \xi_1)$ and $(M_2, \xi_2)$ be cooriented contact manifolds of the same dimension.
We say that an embedding $\varphi \colon M_1 \to M_2$ preserves the shape invariants (or for short, preserves the shape) of two open subsets $U \subset M_1$ and $V \subset M_2$ so that $\overline{U} \subset M_1$ is compact and $\varphi (\overline{U}) \subset V$ if $I_C (U, L, \tau) \subset I_C (V, L, \tau \circ \varphi^*)$ for every $L$ and every homomorphism $\tau \colon H^1 (V, \R) \to H^1 (U, \R)$.
An embedding is said to preserve shape if it preserves the shape of all open subsets $U$ and $V$ as above.
It is said to preserve the modified shape (of two subsets) if $I_C$ is replaced by $\It_C$ in the above definition. \qed
\end{dfn}

\begin{rmk}
By Propositions~\ref{pro:contact-shape} and \ref{pro:tilde-contact-shape}, contact embeddings preserve both shape invariants; of course, for contact embeddings it is again not necessary to make the compactness assumption.
This definition however is preserved under uniform convergence (Proposition~\ref{pro:continuous-contact-shape}).
The restriction of a (modified) shape preserving embedding to an open subset by definition again preserves the (modified) shape.
An embedding that preserves the (modified) shape is contact by Theorem~\ref{thm:contact-shape-preserving}, and thus preserves the (modified) shape of arbitrary subsets. \qed
\end{rmk}

The following analog of Proposition~\ref{pro:conf-symp-shape} is used later to distinguish contact embeddings that preserve coorientation from those that reverse it.

\begin{pro} \label{pro:coorient-contact-shape}
Let $(M_1, \xi_1)$ and $(M_2, \xi_2)$ be cooriented contact manifolds of the same dimension, and $\varphi \colon M_1 \to M_2$ be a contact embedding (that a priori may or may not preserve coorientation).
Then $\varphi$ preserves the (modified) shape invariant if and only if it preserves coorientation.
\end{pro}

\begin{proof}
The argument is the same as in the proof of Proposition~\ref{pro:conf-symp-shape}, except that $c = \pm 1$, there is no freedom of translation, and the reference to Theorem~\ref{thm:torus-shape} is replaced by Propositions~\ref{pro:coisotropic-darboux-weinstein} and \ref{pro:shape-unit-cotangent-bundle}.
\end{proof}

The analog of Proposition~\ref{pro:continuous-symp-shape} is the following continuity property of the contact shape invariants.
The proof is verbatim the same and thus is omitted.

\begin{pro} \label{pro:continuous-contact-shape}
Let $(M_1, \xi_1)$ and $(M_2, \xi_2)$ be cooriented contact manifolds of the same dimension.
Suppose that $\varphi_k \colon M_1 \to M_2$ is a sequence of embeddings that converges uniformly on compact subsets to another embedding $\varphi \colon M_1 \to M_2$, and that $\varphi_k$ preserves (modified) shape for every $k$.
Then $\varphi$ preserves (modified) shape.
\end{pro}

\section{$C^0$-characterization of contact embeddings} \label{sec:contact}
This section contains the proof of Theorem~\ref{thm:contact-shape-preserving} and its corollaries.

\begin{rmk}
For the earlier parts of the present section (up to but not including Theorem~\ref{thm:rig-coiso-hom-class}), we have to ignore coorientation.
In particular, a contact embedding may or may not preserve coorientation. \qed
\end{rmk}

\begin{lem} \label{lem:coisotropic}
Let $(M_1, \xi_1)$ and $(M_2, \xi_2)$ be contact manifolds of the same dimension.
Then an embedding $\varphi \colon M_1 \to M_2$ is contact if and only if it preserves coisotropic submanifolds.
The latter means that the image $\varphi (C)$ is coisotropic whenever $C$ is a coisotropic submanifold.
The same statement holds when restricted to embedded coisotropic tori that are contained in an element of any given open cover of $M_1$.
\end{lem}

\begin{proof}
That contact embeddings preserve coisotropic submanifolds is obvious.
We will prove the converse.
In fact, we will prove that if $\varphi$ is not contact at $x$, then there exists a coisotropic submanifold $C$ through $x$ so that $d\varphi (T_x C) \subset (\xi_2)_{\varphi (x)}$.
Our arguments are local in nature, and therefore also apply to contact manifolds that are not coorientable.

Let $\xi_1 = \ker \alpha_1$ and $\xi_2 = \ker \alpha_2$, and define a (not necessarily nowhere vanishing) function $f$ and a one-form $\beta$ on $M_1$ by $f = (\varphi^* \alpha_2) (R_1)$ and $\beta = (\varphi^* \alpha_2) - f \alpha_1$, where $R_1$ denotes the Reeb vector field associated to $\alpha_1$.
Then $\varphi^* \alpha_2 = f \alpha_1 + \beta$ and $\beta (R_1) = 0$.
By non-degeneracy of $d\alpha_1 |_{\xi_1}$ there exists a unique vector field $X$ on $M_1$ that is tangent to $\xi_1$ and such that $\beta = d\alpha_1 (X, \cdot)$ (see Exercise~3.54 in \cite{mcduff:ist98}).

Suppose that $\varphi$ is not contact at a point $x \in M_1$, or equivalently, the vector $v = X (x) \not= 0$ (by the contact condition, $f$ and $X$ cannot vanish simultaneously).
Let $w \in (\xi_1)_x$ be a vector such that $d\alpha_1 (v, w) = 1$.
By Lemma~\ref{lem:existence-coisotropic-tori}, there exists a coisotropic torus $C$ through $x$ that is tangent to both $R_1 (x) - f w$ and $v$, and so that $T_x C / \langle R_1 (x) - f w \rangle \subset (\xi_1)_x$.
But then $(\varphi^* \alpha_2) |_C = 0$ at the point $x$, and thus $\varphi (C)$ is not transversal to $\xi_2$.
\end{proof}

\begin{rmk}
The same proof applies to two contact structures $\xi_1$ and $\xi_2$ on the same smooth manifold $M$.
That is, $\xi_1 = \xi_2$ if and only if every coisotropic submanifold with respect to $\xi_1$ is also a coisotropic submanifold with respect to $\xi_2$.
This statement continuous to hold when restricted to embedded coisotropic tori in an element of any given open cover of $M$.
This fact is not needed anywhere in this paper but is stated for the sake of completeness. \qed
\end{rmk}

We again record the converse statement in a separate lemma.

\begin{lem}
Let $(M_1, \xi_1)$ and $(M_2, \xi_2)$ be contact manifolds of the same dimension.
Suppose an embedding $\varphi \colon M_1 \to M_2$ is not contact at $x \in M_1$, and let $U \subset M_1$ be a neighborhood of $x$.
Then there exists a coisotropic embedding $\iota \colon T^n \hookrightarrow M_1$ through $x$ whose image is contained in $U$, and so that $\varphi \circ \iota \colon T^n \hookrightarrow M_2$ is not coisotropic (at the point $\varphi (x)$).
The embedding $\widehat {\varphi \circ \iota}_f \colon T^n \hookrightarrow M_2 \times \R_+$ given by $x \mapsto ((\varphi \circ \iota) (x), f (x))$ is thus not Lagrangian for any (positive) function $f$ on $T^n$ and any contact forms on $(M_1, \xi_1)$ and $(M_2, \xi_2)$ defining the symplectizations.
\end{lem}

\begin{rmk}
In order to put the forthcoming argument that completes the proof of Theorem~\ref{thm:contact-shape-preserving} in perspective, assume that $C$ is an (embedded, closed, and connected) $n$-dimensional submanifold of a contact manifold $(M, \xi)$ that is not coisotropic.
Then any lift $L$ of $C$ to a symplectization of $M$ is not Lagrangian.
This applies in particular to $C \times 1 \subset M \times \R_+$.
By Laudenbach-Sikorav's Theorem~\ref{thm:displacement}, there exists a Hamiltonian vector field that is nowhere tangent to $L$.
If $C$ is transversal to $\xi$, then the bundle $E = (\xi |_C)^\perp \oplus \langle R \rangle$ over $C$ is (up to identification of $C$ with $C \times 1$) precisely the symplectic orthogonal complement of $T (C \times 1)$ in $T (M \times \R_+)$.
Then the proof by Laudenbach-Sikorav applies directly in the contact setting, without the need to lift to the symplectization.
The contact vector field $X_F$ so derived of course lifts to the Hamiltonian vector field generated by the Hamiltonian function $t \, \pi^* F$.
Let $X$ be a nowhere vanishing section of $E$ so that $dF (X) > 0$ (see the sketch of the proof of Theorem~\ref{thm:displacement} above).
In the contact setting, $dF = dF (R) \, \alpha - d\alpha (X_F, \cdot)$, so that in contrast to the proof of Theorem~\ref{thm:displacement}, this last condition alone does not necessarily imply that $X_F$ is nowhere tangent to $C$.
A more contact topological proof is therefore required for Theorem~\ref{thm:contact-shape-preserving}. \qed
\end{rmk}

\begin{dfn}
A submanifold $C$ of a contact manifold is called immediately displaceable if there exists a contact vector field $X_F$ (defined in a neighborhood of $C$) that is nowhere tangent to $C$. \qed
\end{dfn}

Recall that if $M$ has dimension $3$, such a surface is called convex, see \cite{geiges:ict08} and the references therein.
We will take advantage of the following known facts.
An embedded hypersurface $S \subset M$ is immediately displaceable if and only if there exists an embedding of $S \times \R$ into $M$ that restricts to the inclusion of $S$ on $S \times 0$ and pulls back the contact structure on $M$ to a vertically invariant contact structure on a neighborhood of $S \times 0$ \cite[Lemma~4.6.19]{geiges:ict08} (the proof given there for surfaces in contact $3$-manifolds applies verbatim to hypersurfaces in higher dimensions).
Moreover, a surface $S$ is convex if and only if its characteristic foliation is divided by a collection of embedded circles \cite[Theorem~4.8.5~(a)]{geiges:ict08} (see below for the definitions).
We will prove a generalization of this result to higher dimensions in the course of the proof of Lemma~\ref{lem:imm-displ} below.

\begin{rmk}
Just like Lagrangian submanifolds, a closed coisotropic submanifold $C$ can never be immediately displaceable, and in fact, the argument is quite similar.
Let $F$ be a smooth function defined near $C$, and denote by $X_F = F R + Y_F$ its contact vector field, where $Y_F$ is tangent to $\xi$ and $R$ is a Reeb vector field that is tangent to $C$ (see Lemma~\ref{lem:description-coisotropic}).
Since $C$ is closed, $F$ must have a critical point $x \in C$.
Then $d\alpha (Y_F (x), v) = - dF (v) = 0$ for all $v \in \xi_x \cap T_x C$, i.e.\ $Y_F (x) \in (\xi_x \cap T_x C)^\perp$.
But the latter is Lagrangian, so that $Y_F (x)$ must be tangent to $C$, and since $R (x)$ is also tangent to $C$, the claim follows. \qed
\end{rmk}

\begin{exa}
By Theorem~\ref{thm:displacement}, a closed $n$-dimensional submanifold $L$ of $W$ is either Lagrangian or immediately displaceable.
In contrast, a closed $n$-dimensional submanifold $C$ of $M$ can be neither coisotropic nor immediately displaceable, as this example shows.
In fact, the example can be generalized to any contact manifold.

Consider the contact manifold $M = T^2 \times \R$ with coordinates $(\varphi, \theta)$ on $T^2$ and $z$ on $\R$, and contact form $\alpha = d\theta + z d\varphi$.
Let $\zeta \colon S^1 \to S^1$ be a smooth function.
The surface $S = \{ z = \epsilon \sin (\zeta (\theta)) \}$ is a graph over $T^2 \times 0$, so that we may use the coordinates $(\varphi, \theta)$ on $S$.
If the function $\zeta$ is constant, then $S$ is coisotropic \cite[Example~4.8.4~(2)]{geiges:ict08}.
(The case $\zeta = 0$ is equivalent to the choice $\epsilon = 0$.)
On the other hand, if $\zeta$ is strictly monotone (i.e.\ its derivative $\zeta'$ nowhere vanishes), then $S$ is convex \cite[Example~4.8.10]{geiges:ict08} and \cite[Figure~4.42 (left)]{geiges:ict08} (with $\zeta$ the identity function).
However, if $\zeta$ is constant on some open subset and strictly monotone on another open subset, then $S$ is neither coisotropic nor convex. \qed
\end{exa}

Before stating the main lemma, we provide further details on the aforementioned notions.
Let $S \subset M$ be an embedded oriented hypersurface.
At a point $x \in S$, the vector space $(\xi_x \cap T_x S)^\perp$ is either $\{ 0 \}$ or a one-dimensional subspace of $\xi_x$ that is contained in $\xi_x \cap T_x S$.
The characteristic foliation $\F$ of $S$ (with respect to the contact structure $\xi$) is the singular one-dimensional foliation defined by the distribution $(\xi |_S \cap TS)^\perp$ (with the orientation established below).
In a neighborhood of $S$ (which is identified with $S \times \R$ so that $S$ corresponds to $S \times 0$), a contact form (that defines $\xi$) can be written as $\alpha = \beta_t + u_t \, dt$, where $t$ denotes the coordinate on $\R$, $\beta_t$ is a smooth family of one-forms and $u_t$ a smooth family of functions on $S$.

\begin{dfn}
Let $\Omega$ be a volume form on an embedded oriented hypersurface $S \subset M$.
The characteristic foliation is defined by the vector field $X$ that satisfies $\iota_X \Omega = \beta_0 \wedge (d\beta_0)^{n - 2}$, with the orientation provided by $X$. \qed
\end{dfn}

Here $\iota_X \Omega$ denotes interior multiplication of $\Omega$ with $X$.
The vector field $\lambda X$, where $\lambda$ is a positive function on $S$, defines the same oriented foliation, and $X$ is unique up to this form of rescaling.
See section~2.5.4 in \cite{geiges:ict08} for details.

\begin{dfn}
Let $S$ be an embedded closed surface in a contact $3$-manifold.
A collection $\Gamma$ of embedded circles (in $S$) is said to divide the characteristic foliation $\F$ of $S$ if $\Gamma$ is transverse to $\F$, and there exists an area form $\Omega$ on $S$ and a vector field $X$ that defines $\F$, so that $\mathcal L_X \Omega \not= 0$ on $S \backslash \Gamma$, and the vector field $X$ points out of $S_+ = \{ x \in S \mid \div_\Omega (X) (x) > 0 \}$ along $\Gamma$. \qed
\end{dfn}

Here $\mathcal L_X \Omega$ denotes the Lie derivative and $\div_\Omega (X)$ the divergence.
See section~4.8 in \cite{geiges:ict08} and in particular \cite[Definition~4.8.3]{geiges:ict08} for details.

\begin{lem} \label{lem:imm-displ}
Let $(M_1, \xi_1)$ and $(M_2, \xi_2)$ be contact manifolds of dimension $2 n - 1$.
Suppose an embedding $\varphi \colon M_1 \to M_2$ is not contact at $x \in M_1$, and let $U \subset M_1$ be a neighborhood of $x$.
Then there exists a coisotropic embedding $\iota \colon T^n \hookrightarrow M_1$ through $x$ whose image is contained in $U$, and so that (the image of) the embedding $\varphi \circ \iota \colon T^n \hookrightarrow M_2$ is immediately displaceable.
\end{lem}

The notions of non-Lagrangian, transversal, and immediately displaceable (and in particular, convex) are all generic, so the heart of the argument is really a matter of carefully choosing the starting coisotropic embedding.

\begin{proof}
The machinery for the aforementioned contact topological proof is mostly developed in dimension $3$, so we handle that case first.

Let $\iota \colon T^2 \hookrightarrow M_1$ be an embedded coisotropic torus; by Lemma~\ref{lem:existence-coisotropic-tori}, these exist in abundance (in a sense made precise there), and we will successively modify this embedding to prove the present lemma.
To simplify notation, we identify $\iota$ with its image $\iota (T^2) = C$, and write $\partial_z$ and $\partial_\theta$ for the vector fields $\iota_* (\partial / \partial z)$ and $\iota_* (\partial / \partial \theta)$, respectively, where $(z, \theta)$ are coordinates on $T^2$.
By the construction of such tori in Lemma~\ref{lem:existence-coisotropic-tori}, we may assume that $\partial_z = R_1$ is the restriction of the Reeb vector field of a contact form $\alpha_1$ on $M_1$ to $C$.
Moreover, we identify the image $\varphi (C)$ with $C$ as a submanifold of $(M_1, \varphi^* \xi_2)$, and to simplify notation further, we write $\xi_2$ for the contact structure $\varphi^* \xi_2$ on $M_1$, and $\alpha_2$ for the pull-back $\varphi^* \alpha_2$ of a contact form on $M_2$.
Our argument thus takes place entirely on the manifold $M_1$ with two different contact structures $\xi_1$ and $\xi_2$ (and with a single submanifold $C$, which is coisotropic with respect to $\xi_1$).

Since we have the freedom to choose the submanifold $C$, we may assume that $\partial_z$ is transverse to both $\xi_1$ and $\xi_2$.
The necessary argument is entirely analogous to the construction of a transverse knot; the circle embedding can be chosen to be transverse to two (hyper-)plane bundles that are nowhere integrable.
Let $\alpha_1$ be a contact form with $\ker \alpha_1 = \xi_1$ so that its Reeb vector field restricts to $\partial_z$ along the transverse knot $S^1$.
This contact form will be fixed for the remainder of the proof.
If $\alpha_2$ is any contact form that defines $\xi_2$, then $\alpha_2 = f \alpha_1 + \beta = f \alpha_1 + d\alpha_1 (X, \cdot)$, with $X$ tangent to $\xi_1$, see the proof of Lemma~\ref{lem:coisotropic}.
By assumption, the function $f$ does not vanish along $S^1$, and by continuity, it is non-vanishing in a neighborhood $U$ of $S^1$.
We may assume that $C \subset U$, again see Lemma~\ref{lem:existence-coisotropic-tori}.
By reversing the orientation of $S^1$ if necessary, we therefore have $f > 0$ in a neighborhood of $C$.

As a consequence of positivity of $f$, there exists a unique smooth function $v$ on $C$ so that the vector field $Y = \partial_\theta - v \, \partial_z$ is tangent to $\xi_2$ everywhere on $C$.
In fact, this function is given by $v = \alpha_1 (\partial_\theta) + d\alpha_1 (X, \partial_\theta) / f$, and $\partial / \partial z \, (\alpha_1 (\partial_\theta)) = 0$.
The vector field $Y$ defines the characteristic foliation $\F$ of $C$ with respect to $\xi_2$, which in this case is non-singular.
Let $\lambda \colon C \to \R$ denote a positive function.
The divergence of the vector field $\lambda Y$ with respect to the area form $\Omega = dz \wedge d\theta$ on $C$ is $\partial \lambda / \partial \theta - v \, \partial \lambda / \partial z - \lambda \, \partial v / \partial z$.
By assumption, $\varphi$ does not preserve the plane bundle $(\xi_1)_x$, so the vector $X (x)$ is non-zero.
By shrinking $U$ if necessary, we may assume that $X$ is nowhere vanishing on $U$.
We may then choose the Lagrangian circle $L = S^1$ (in the notation of Lemma~\ref{lem:existence-coisotropic-tori}, $C = S^1 \times L$) so that $v$ is not locally constant anywhere on $C$.
In fact, for generic choices of $L$ and $\lambda$, the zeroes of the divergence of $\lambda Y$ are non-degenerate, and thus $\div_\Omega (\lambda Y)$ vanishes only on a collection $\Gamma$ of isolated embedded circles, and moreover, $\Gamma$ is transverse to $\F$.
In other words, the collection of embedded circles $\Gamma$ divides the characteristic foliation of $C$, and hence $C$ is convex.

Before giving the proof in the case $\dim M > 3$, we provide details of the proof of Theorem~4.8.5~(a) in \cite{geiges:ict08} that are relevant for generalizing the argument to arbitrary dimensions.
Let $\beta = \iota_Z \Omega$, where $Z = \lambda Y$ and $\Omega = dz \wedge d\theta$ are as above, and define $\alpha = \beta + u \, dt$, where $t$ denotes the coordinate on $\R$ and $u$ is a smooth function on $S$.
It suffices to show that $\alpha$ is a contact form on a neighborhood of $S \times 0$ (which is again identified with $S$) in $S \times \R$, which in turn is equivalent to the condition that $u \, \div_\Omega (Z) - du (Z) > 0$.
The latter is satisfied away from $\Gamma$ with $u = \pm 1$.
One can use the flow of $Z$ to identify a neighborhood of $\Gamma$ with $\Gamma \times [- \epsilon, \epsilon]$ so that $\Gamma$ corresponds to $\Gamma \times 0$.
Then the function $u (p, s) = g (p, s) \cdot h (p, s)$, where $h (p, s) = \exp \left( - \int_0^s \div_\Omega (Z) (p, r) \, dr \right)$, and $g$ satisfies $\partial g / \partial s > 0$ and $g (p, s) = \pm 1 / h (p, s)$ near $s = \pm \epsilon$, satisfies the above contact condition.

Now suppose that $\dim M > 3$ (so that $n > 2$).
Let $C$ be an embedded torus as in the above proof in the dimension $3$ case, and extend $C$ to a coisotropic torus $T^n = S^1 \times T^{n - 1} = C \times T^{n - 2}$ inside $S^1 \times \R^{2 n - 2}$.
We continue to write $(z, \theta)$ for coordinates on $C$.
Denote by $W_{(z, \theta)}$ the ($2 n - 4$)-dimensional subspace of the symplectic orthogonal complement of $\partial_\theta$ at $(z, \theta)$ that is linearly independent of $\partial_\theta$.
Then the above $T^{n - 2}$-factor can be chosen as (or tangent to) any Lagrangian subspace of $W$.
Strictly speaking, the resulting coisotropic torus $T^n = C \times T^{n - 2}$ is a fibered product over $C$, but for simplicity of notation we disregard this subtlety.
The fibered product $S = C \times W$ (or $C \times V$, where $V \subset W$ is a neighborhood of the torus $T^{n - 2}$) is then an oriented embedded hypersurface of $S^1 \times \R^{2 n - 2}$, whose characteristic foliation induced by $\xi_2$ is precisely given by the vector field $Z$ that we constructed in the dimension $3$ case.
(The hypersurface $S$ is of course not compact, but our constructions are all local near $T^n$, and thus this issue does not affect our arguments.)
The set $\Gamma$ of zeroes of its divergence (with respect to an appropriate volume form $\Omega$ on $S$) is a collection of isolated embedded codimension one submanifolds of $S$ (of the form $S^1 \times V$, where $S^1$ belongs to the dividing set in the dimension $3$ case) that are transverse to $\F$.

Let $\beta$ be the restriction of a contact form for $\xi_2$ to $S$ so that $\iota_Z \Omega = \beta \wedge (d\beta)^{n - 2}$ (see the remarks before this lemma).
As before, define a vertically invariant one-form $\alpha = \beta + u \, dt$ on $S$.
In this case, the contact condition for $\alpha$ translates into $u \, \div_\Omega (Z) - (n - 1) du (Z) > 0$ (the computation is step-by-step the same as in dimension $3$, see \cite{geiges:ict08}, and in fact is partly carried out there, so we omit lengthy details).
But then the argument in the dimension $3$ case goes through almost verbatim, except that the exponent in the definition of $h$ must be divided by $n - 1$, and the second condition on $g$ has to be replaced by the requirement that $g (p, s)$ equals the reciprocal of $\pm (n - 1) h (p, s)$ near $s = \pm \epsilon$.
Thus $\alpha$ defines a vertically invariant contact structure on a neighborhood of $S \times 0$ that coincides with $\xi_2$ on $S$, which is equivalent to the existence of a contact vector field that is transverse to $S$.
In particular, this contact vector field is transverse to $T^n$ (which is identified with $(\varphi \circ \iota) (T^n)$), and the proof of the lemma (for arbitrary dimension) is complete.
\end{proof}

\begin{rmk}
Denote by $\Psi$ the embedding of a neighborhood of $S \times 0$ into $M$ from the preceding proof.
Then the vertical vector field $X_F = \Psi_* (\partial / \partial t)$ that is transverse to $T^n$ is in fact strictly contact (i.e.\ it preserves the contact form $\alpha$ on $(M, \xi)$ so that $\Psi^* \alpha = \beta + u \, dt$).
The lift of $X_F$ to a Hamiltonian vector field on the symplectization (with respect to $\alpha$) is then trivial in the $\R$-direction. \qed
\end{rmk}

The following proposition from \cite{mueller:hdp17} is a contact analog of the Laudenbach-Sikorav theorems from section~\ref{sec:symp}, and it is proved similarly by constructing a non-constant holomorphic disk with prescribed coisotropic boundary conditions.

\begin{thm}[{\cite[Theorem~1.3]{mueller:hdp17}}] \label{thm:rig-coiso-hom-class}
Let $\iota \colon L \hookrightarrow (\R^{2 n - 1}, \xi_0)$ be an embedding of a (closed and connected) $n$-dimensional manifold that is immediately displaceable.
Then there exists a neighborhood $N$ of $\iota (L)$ that does not admit any coisotropic embeddings $\jmath \colon L \hookrightarrow N$ so that the homomorphism $\jmath_* \colon H_1 (L, \R) \to H_1 (N, \R)$ is injective.
In particular, the modified shape $\It_C (N, L, \iota^*)$ is empty.
\end{thm}

\begin{rmk}
There is no immediate counterpart of Theorem~\ref{thm:rig-lag-period} for coisotropic embeddings.
Although the property of being rational is preserved by rescaling, the size of the generator $\gamma$ rescales by the same factor.
More importantly, the latter depends on more than just the embedding itself \cite{mueller:hdp17}.
The analog of Theorem~\ref{thm:rig-lag} (with the assumption of uniform convergence) can of course be proved similarly. \qed
\end{rmk}

If a compact submanifold $L$ of a contact manifold is immediately displaceable, then (a neighborhood of) the cone over $L$ is also immediately displaceable.
That observation gives rise to a more direct proof of Theorem~\ref{thm:rig-lag-hom-class} in the case $L = T^n$, which also applies to the original contact shape invariant.

\begin{thm} \label{thm:rig-coiso-arnold-conjecture}
Let $\iota \colon T^n \hookrightarrow M$ be an embedding that is immediately displaceable.
Then there exists a neighborhood $N$ of $\iota (T^n)$ that does not admit any coisotropic embeddings $\jmath \colon T^n \hookrightarrow N$ so that the homomorphism $\jmath_* \colon H_1 (T^n, \R) \to H_1 (N, \R)$ is injective.
In particular, the modified shape $\It_C (N, \iota^*)$ is empty.
In fact, there exists no Lagrangian embedding of $T^n$ into the cone $N \times \R_+$ so that the induced homomorphism $\jmath_* \colon H_1 (T^n, \R) \to H_1 (N, \R)$ is injective, and in particular, the shape $I_C (N, \iota^*)$ is also empty.
\end{thm}

\begin{proof}
The Arnold conjecture holds for the symplectization of a compact contact manifold, and thus the argument in Remark~\ref{rmk:arnold-conjecture} applies verbatim.
\end{proof}

\begin{rmk} \label{rmk:coiso-arnold-conjecture}
As in Remark~\ref{rmk:arnold-conjecture}, the previous argument extends to arbitrary (closed and connected) manifolds $L$ under the same additional hypotheses on the homotopy type of the coisotropic embedding, and one can again define a suitable contact shape invariant for that purpose. \qed
\end{rmk}

We are finally in a position to complete the proof of Theorem~\ref{thm:contact-shape-preserving}.

\begin{proof}[Proof of Theorem~\ref{thm:contact-shape-preserving}]
We have already verified in Propositions~\ref{pro:contact-shape} and \ref{pro:tilde-contact-shape} that contact embeddings preserve the (modified) shape invariants, so we only need to prove the converse.

By Proposition~\ref{pro:coorient-contact-shape}, we may assume that $\varphi$ is not a contact embedding that reverses coorientation.
As in the proof of Theorem~\ref{thm:symp-shape-preserving}, suppose that $\varphi$ is not contact at $x \in B_r^{2 n - 1}$, and again assume without loss of generality that $W = \R^{2 n - 1}$ with its standard contact structure.
By Lemma~\ref{lem:imm-displ}, there exists a coisotropic embedding $\iota \colon T^n \hookrightarrow B_r^{2 n - 1}$ so that $\varphi \circ \iota \colon T^n \hookrightarrow \R^{2 n - 1}$ is immediately displaceable.
Let $U$ be a neighborhood of $\iota (T^n)$ so that $\It_C (N, \varphi \circ \iota^*)$ is empty, where $\overline{U} \subset B_r^{2 n - 1}$ is compact and $(\varphi \circ \iota) (\overline{U}) \subset N$; this exists by Theorem~\ref{thm:rig-coiso-hom-class} or Theorem~\ref{thm:rig-coiso-arnold-conjecture}.
But $\It_C (U, \iota^*)$ contains at least the (oriented line represented by the) vector $a = [\iota^* \alpha_0]$ (and in fact, by Proposition~\ref{pro:coisotropic-darboux-weinstein}, we may choose the neighborhood $U$ of $\iota (T^n)$ so that $\It_C (U, \iota^*) = A$, where $A \subset S^{n - 1}$.
Therefore $\varphi$ does not preserve the modified shape invariant.
The proof for the original shape invariant $I_C$ is analogous.
\end{proof}

\begin{proof}[Proof of Corollary~\ref{cor:rig-contact-emb}]
The proof is a carbon copy of the proof of Corollary~\ref{cor:rig-symp-emb} with Theorem~\ref{thm:symp-shape-preserving} and Proposition~\ref{pro:symp-shape} replaced by Theorem~\ref{thm:contact-shape-preserving} and Proposition~\ref{pro:contact-shape} or \ref{pro:tilde-contact-shape}.
If the embeddings $\varphi_k$ are not assumed to preserve coorientation, we may pass to a subsequence and if necessary argue as we did for anti-symplectic embeddings.
\end{proof}

\begin{proof}[Proof of Corollary~\ref{cor:rig-contact-diff}]
The proof is virtually the same as the proof of Corollary~\ref{cor:rig-symp-diff} with Corollary~\ref{cor:rig-symp-emb} replaced by Corollary~\ref{cor:rig-contact-emb}.
If $\xi$ is not coorientable or if the diffeomorphisms $\varphi_k$ are not assumed to preserve coorientation, we may again pass to a subsequence and if necessary argue as we did for anti-symplectic embeddings.
\end{proof}

\begin{cor}
Let $\varphi_k \colon B_r^{2 n - 1} \to M$ be a sequence of contact embeddings that reverse coorientation that converges uniformly on compact subsets to an embedding $\varphi \colon B_r^{2 n - 1} \to M$.
Then $\varphi$ is contact but reverses coorientation.
\end{cor}

\begin{proof}
One can either argue as for anti-symplectic embeddings using Corollary~\ref{cor:rig-contact-emb}, or use the final corollary of this section and the remark after it.
\end{proof}

\begin{cor}
Let $\varphi_k \colon M \to M$ be contact diffeomorphisms that reverse the given coorientation and converge uniformly on compact subsets to another diffeomorphism $\varphi \colon M \to M$.
Then $\varphi$ is contact but reverses the coorientation.
In other words, the set of diffeomorphisms that preserve $\xi$ but reverse its coorientation is $C^0$-closed in the group $\Diff (M)$ of all diffeomorphisms.
\end{cor}

\begin{proof}
Indeed, if this set is non-empty, then it coincides with $\varphi \cdot \Diff_+ (M,\xi)$ for some diffeomorphism $\varphi$ that reverses the coorientation of $\xi$.
Alternatively, one may give a proof using the previous corollary.
\end{proof}

\begin{cor}
An embedding $\varphi \colon B_r^{2 n - 1} \to M$ preserves the contact structures but reverses coorientation if and only if it reverses the shape invariant.
\end{cor}

\begin{proof}
This follows from Theorem~\ref{thm:contact-shape-preserving} by the same argument as for anti-symplectic embeddings, or directly along the same lines as the proof of Theorem~\ref{thm:contact-shape-preserving}.
\end{proof}

\begin{rmk}
Reversing shape is again a property that is preserved by uniform convergence on compact subsets.
The proof is verbatim the same as the proof of Proposition~\ref{pro:continuous-contact-shape}. \qed
\end{rmk}

\begin{rmk}
For all results that ignore coorientation, one may instead work with the version of the (modified) shape that ignores coorientation, see Remark~\ref{rmk:shape-no-coorientation}. \qed
\end{rmk}

\section{Contact forms and strictly contact embeddings} \label{sec:strictly-contact}
The present section is concerned with (contact) embeddings and diffeomorphisms that preserve given contact forms.
Recall that the contact condition for $\xi = \ker \alpha$ can be expressed as $\Omega_\alpha = \alpha \wedge (d\alpha)^{n - 1} \not= 0$.

Suppose that $(M_1, \xi_1 = \ker \alpha_1)$ and $(M_2, \xi_2 = \ker \alpha_2)$ are (cooriented) contact manifolds of the same dimension, and $\varphi \colon M_1 \to M_2$ is a (coorientation preserving) contact embedding.
Then $\varphi^* \alpha_2 = f \alpha_1$ for a positive function $f$ on $M_1$, and thus $\varphi^* \Omega_{\alpha_2} = f^n \, \Omega_{\alpha_1}$.
In particular, the contact embedding $\varphi$ also preserves the contact forms if and only if it in addition preserves the induced volume forms.
Therefore the following theorem is an immediate corollary of Theorem~\ref{thm:contact-shape-preserving}.

\begin{thm} \label{thm:strictly-contact-shape-preserving}
An embedding $\varphi \colon B_r^{2 n - 1} \to M$ is strictly contact if and only if it preserves the (modified) shape invariant and preserves volume.
\end{thm}

\begin{proof}
If $\varphi$ is strictly contact then it preserves the (modified) shape and the induced volume forms.
Conversely, if $\varphi$ preserves the (modified) shape, then it is contact by Theorem~\ref{thm:contact-shape-preserving}, and since it also preserves volume, it must be strictly contact.
\end{proof}

\begin{cor}
Let $\varphi_k \colon B_r^{2 n - 1} \to M$ be a sequence of strictly contact embeddings that converges uniformly on compact subsets to an embedding $\varphi \colon B_r^{2 n - 1} \to M$.
Then $\varphi$ is strictly contact, that is, $\varphi^* \alpha = \alpha_0$.
\end{cor}

\begin{proof}
By Proposition~\ref{pro:continuous-contact-shape}, the limit $\varphi$ preserves the (modified) shape invariant.
Each $\varphi_k$ preserves the measures induced by the volume forms $\Omega_{\alpha_0}$ and $\Omega_\alpha$ (which is a Radon measure if $M$ is not compact), and this property is also preserved by uniform convergence on compact subsets.
But a smooth map preserves measure if and only if it preserves the corresponding volume forms, and thus $\varphi$ is volume preserving.
Then by the previous theorem, $\varphi$ is a strictly contact embedding.
\end{proof}

One could also argue using Corollary~\ref{cor:rig-contact-emb} to prove this corollary.
Similarly, one may use either the corollary we just proved or Corollary~\ref{cor:rig-contact-diff} to prove the next result.
At this point the proof is straightforward and therefore omitted.

\begin{cor}[\cite{ms:gae14}]
The group $\Diff (M, \alpha)$ of strictly contact diffeomorphisms is closed in the group $\Diff (M)$ of diffeomorphisms of $M$ in the $C^0$-topology.
That is, if $\varphi_k \colon M \to M$ is a sequence of strictly contact diffeomorphisms that converges uniformly on compact subsets to a diffeomorphism $\varphi \colon M \to M$, then $\varphi$ is strictly contact, i.e.\ $\varphi^* \alpha = \alpha$.
\end{cor}

\begin{rmk}
The proof in \cite{ms:gae14} uses the fact that the lift of a strictly contact diffeomorphism to the symplectization (with respect to the specific contact form $\alpha$) is symplectic and trivial in the second argument.
The previous corollary is then a direct consequence of Corollary~\ref{cor:rig-symp-diff}, and one can argue similarly for strictly contact embeddings. \qed
\end{rmk}

It is tempting to restrict the definition of the (modified) shape invariant to single out strictly contact embeddings, and based on that invariant give a proof of the previous three results along the same lines as the arguments in section~\ref{sec:contact}.
Here is an ad hoc adaptation of Definition~\ref{dfn:tilde-contact-shape} to an invariant of a contact form $\alpha$.

\begin{dfn} \label{dfn:strictly-contact-shape}
The strictly contact $(\alpha, L, \tau)$-shape $\It_{SC} (M, \alpha, L, \tau) \subset H^1 (L, \R)$ of $M$ is by definition the set of all $z \in H^1 (L, \R)$ such that there exists a coisotropic embedding $\iota \colon L \hookrightarrow M$ so that $\iota^* = \tau$, the one-form $\iota^* \alpha$ is closed and defines the codimension one distribution $\iota^* \xi$, and $z = [\iota^* \alpha]$. \qed
\end{dfn}

\begin{rmk}
The coisotropic embeddings as in this definition lift to Lagrangian embeddings into $M \times 1 \subset M \times \R_+$, where the latter denotes the symplectization with respect to the contact form $\alpha$.
Clearly a strictly contact embedding preserves the strictly contact shape.
However, coisotropic is really a concept related to a contact structure, not a contact form, and embeddings as in the previous definition are not as abundant in general (or fail to exist) to make the line of argument in section~\ref{sec:contact} go through.
For example, the Reeb vector field of the standard contact form $\alpha_0$ on $\R^{2 n - 1}$ is $\partial / \partial z$, and there cannot exist a closed submanifold of $\R^{2 n - 1}$ that is everywhere tangent to $\partial / \partial z$.
Therefore $\It_{SC} (\R^{2 n - 1}, \alpha_0, L, \tau)$ is always empty, and there is no hope to prove Theorem~\ref{thm:strictly-contact-shape-preserving} based on the strictly contact shape. \qed
\end{rmk}

There are some special cases however where the situation is more promising.
We prove the next lemma just for the sake of completeness.
It applies for instance to $S^1 \times \R^{2 n - 2}$ with the contact form considered in section~\ref{sec:contact}, the unit cotangent bundle $S T^* T^n$ with its standard contact form, and to regular contact manifolds (i.e.\ the Reeb vector field induces a free $S^1$-action on $M$).

\begin{lem}
Let $(M, \xi = \ker \alpha)$ be a contact manifold such that for each point $x \in M$ and each vector $v \in \xi_x$ there exists a coisotropic submanifold $C$ through $x$ that is tangent to $v$, and so that the restriction of $\alpha$ to $C$ is closed.
Suppose further that $f$ is a nowhere vanishing smooth function on $M$ so that the restriction of $f \alpha$ to each such coisotropic submanifold is also closed.
Then $f$ is constant.
In other words, the contact form $\alpha$ with the above property is unique up to rescaling by a constant.
If $(M_1, \xi_1 = \ker \alpha_1)$ and $(M_2, \xi_2 = \ker \alpha_2)$ are two contact manifolds as above of the same dimension, and $\varphi \colon M_1 \to M_2$ is a contact embedding that preserves such coisotropic submanifolds, then $\varphi^* \alpha_2 = c \, \alpha_1$ for a constant $c \not= 0$.
\end{lem}

\begin{proof}
The argument is similar to the discussion before the proof of Theorem~\ref{thm:neighborhood}.
Denote by $R$ the Reeb vector field associated to the contact form $\alpha$.
Since $d (f \alpha) = df \wedge \alpha + f \, d \alpha$, we have by assumption $df (v) = (d (f \alpha) - f \, d \alpha) (v, R) = 0$ for every $v \in \xi$.
That means that for each regular value $c$ of $f$, the preimage $f^{- 1} (c)$ is a codimension one submanifold that is everywhere tangent to $\xi$.
This is of course impossible, so that $f$ possesses no regular values whose preimage is non-empty.
By Sard's Theorem, the image of $f$ then does not contain an open interval, and by continuity, it is comprised of a single point.
The last part of the lemma follows from the first part by considering the contact form $f \alpha_1 = \varphi^* \alpha_2$.
\end{proof}

\begin{rmk}
One can replace the assumption that $\varphi$ is contact in the last part of the lemma by a stronger hypothesis on the existence of the special type of coisotropic submanifolds considered here so that the proof of Lemma~\ref{lem:coisotropic} applies. \qed
\end{rmk}

\section{Shape preserving embeddings of non-exact symplectic manifolds and contact manifolds that are not coorientable} \label{sec:non-exact}
We now extend the definition of shape preserving to symplectic manifolds that are not necessarily exact by making minor changes to Definition~\ref{dfn:symp-shape-preserving}, and likewise for contact manifolds that are not necessarily coorientable.

\begin{dfn}
We call an open subset $U$ of a (not necessarily exact) symplectic manifold $(M, \omega)$ exact if the cohomology class $[\omega]$ belongs to the kernel of the homomorphism $i^* \colon H^1 (M, \R) \to H^1 (U, \R)$, where $i \colon U \to M$ is the inclusion.
By a slight abuse of notation, we call an open subset $U$ of a (not necessarily coorientable) contact manifold $(M, \xi)$ exact if $\xi |_U$ is coorientable. \qed
\end{dfn}

\begin{dfn} \label{dfn:shape-preserving-non-exact}
Let $M_1$ and $M_2$ be (not necessarily exact) symplectic manifolds or (not necessarily coorientable) contact manifolds of the same dimension, and $\varphi \colon M_1 \to M_2$ be an embedding.
Let $U \subset M_1$ and $V \subset M_2$ be two exact open subsets so that $\overline{U} \subset M_1$ is compact, and $\varphi (\overline{U}) \subset V$.
We say that $\varphi$ preserves the (modified) shape invariant of $U$ and $V$ provided that $I (U, L, \tau) \subset I (V, L, \tau \circ \varphi^*)$ (or $I_C (U, L, \tau) \subset I_C (V, L, \tau \circ \varphi^*)$ or $\It_C (U, L, \tau) \subset \It_C (V, L, \tau \circ \varphi^*)$) for every $L$ and for every homomorphism $\tau \colon H^1 (U, \R) \to H^1 (L, \R)$.
An embedding is said to preserve the (modified) shape invariant if it preserves the (modified) shape of all open subsets $U \subset M_1$ and $V \subset M_2$ as above. \qed
\end{dfn}

\begin{rmk}
Recall from section~\ref{sec:symp-shape} that every point in a symplectic manifold has an exact neighborhood, and likewise for contact manifolds, see Remark~\ref{rmk:coorientation}.
The proofs in this paper therefore apply directly to symplectic manifolds that are not necessarily exact and to contact manifolds that are not necessarily coorientable. \qed
\end{rmk}

\section{Shape as sufficient condition for existence \\ of symplectic and contact embedding} \label{sec:shape-preserving}
Proposition~\ref{pro:symp-shape} implies that the shape is an obstruction to the existence of a symplectic embedding, and by Theorem~\ref{thm:symp-shape-preserving}, a given embedding is symplectic if and only if it preserves shape (Definition~\ref{dfn:symp-shape-preserving}).
It is therefore natural to ask, given two exact symplectic manifolds $W_1$ and $W_2$ (of the same dimension) that have the same shape, does this property imply the existence of a symplectic embedding $W_1 \to W_2$?

Recall from Example~\ref{exa:plane} that the shape of an open and connected subset of $(\R^2, \omega_0)$ recovers its area.
It is well known that there exists an area preserving diffeomorphism between two such subsets if and only if the sets are diffeomorphic and have the same (total) area.
Both of the last two conditions are obviously necessary.
The present section discusses similar results in higher dimensions.

The following theorem is due to V.~Benci and Sikorav \cite{sikorav:rsc89}.
We will give an elegant proof in which the difficult parts of the argument are hidden within the previously established properties of the shape invariant.

\begin{thm} [\cite{sikorav:rsc89}] \label{thm:cotangent-symp-rigidity}
Let $U \subset U'$ and $V \subset V'$ be open and connected subsets of $\R^n$ such that $H^1 (U', \Z)$ and $H^1 (V', \Z)$ are trivial.
Then the following three statements are equivalent:
\begin{enumerate}
\item there exists a symplectic diffeomorphism $\varphi \colon T^n \times U' \to T^n \times V'$ that maps $T^n \times U$ to $T^n \times V$,
\item there exists a diffeomorphism $\varphi \colon T^n \times U' \to T^n \times V'$ that maps $T^n \times U$ to $T^n \times V$, and $I (T^n \times U, \iota_0^*) = I (T^n \times V, \iota_0^* \circ \varphi^*)$ (up to translation), and
\item there exists a unimodular matrix $A \in GL (n, \Z)$ and a vector $b \in \R^n$ such that $V = A U + b = A (U + A^{- 1} b)$.
\end{enumerate}
In fact, if statement (1) or (2) holds, the matrix $A$ represents the homomorphism $H^1 (T^n, \Z) \to H^1 (T^n, \Z)$ induced by $\varphi$ (under the identification of $H^1 (T^n, \R)$ with $\R^n$), and the translation vector in (2) is given by $A^{- 1} b$, which corresponds to the cohomology class $\iota_0^* ([\varphi^* \lambda_\can - \lambda_\can])$.
Moreover, we may choose $U' = V' = \R^n$.
\end{thm}

\begin{rmk}
In \cite{sikorav:rsc89}, Sikorav proved the theorem with $U = U'$ and $V = V'$.
A proof of the other extreme case $U' = V' = \R^n$ can be found in \cite[Section~11.3]{mcduff:ist98}.
Neither version contains statement (2) above. \qed
\end{rmk}

\begin{proof}
That (1) implies (2) follows immediately from Proposition~\ref{pro:symp-shape}.

By the hypotheses on $U'$ and $V'$, we may identify $H^1 (T^n \times U', \R)$ with $H^1 (T^n, \R)$ via the isomorphism $\iota_0^*$, and likewise for $H^1 (T^n \times V', \R)$.
Assuming statement~(2), the diffeomorphism $\varphi$ induces an isomorphism $\Phi \colon H^1 (T^n, \Z) \to H^1 (T^n, \Z)$ so that, under the above identifications, $\iota_0^* \circ \varphi^* = \Phi \circ \iota_0^*$.
Then by Proposition~\ref{pro:L-diffeos}, the shape satisfies $I (T^n \times V, \iota_0^* \circ \varphi^*) = I (T^n \times V, \Phi \circ \iota_0^*) = \Phi (I (T^n \times V, \iota_0^*))$.
By Sikorav's Theorem~\ref{thm:torus-shape}, we have $I (T^n \times U, \iota_0^*) = U$ and $I (T^n \times V, \iota_0^*) = V$.
Then by assumption, $U = \Phi (V - b)$ for some vector $b \in \R^n$, and (3) follows with $A = \Phi^{- 1}$.

To see that (3) implies (1), define a diffeomorphism $\varphi \colon T^n \times \R^n \to T^n \times \R^n$ by $(q, p) \mapsto ((A^{-1})^t q, A p + b)$.
Then $\varphi$ is clearly symplectic, and maps $T^n \times U$ to $T^n \times V$.
Moreover, the translation vector is $\iota_0^* ([\varphi^* \lambda_\can - \lambda_\can]) = A^{-1} b$.
\end{proof}

\begin{rmk}
Sikorav's Theorems~\ref{thm:torus-shape} and \ref{thm:cotangent-symp-rigidity} can be generalized to arbitrary (closed and connected) $n$-dimensional manifolds $L$.
The statements are however more cumbersome than in the case of tori due to the lack of the natural product structure $T^* T^n = T^n \times \R^n$ for general manifolds $L$.
(For submanifolds $L \subset \R^n$, the cotangent bundle $T^*L$ can be naturally identified with a quotient of $L \times \R^n$, see \cite[Exercise~11.22]{mcduff:ist98}.)
In that case, one should consider subsets $A \subset H^1 (L, \R^n)$ that are comprised of the union of the images of harmonic sections $L \to T^* L$ (with respect to some auxiliary Riemannian metric).
For the proof of Theorem~\ref{thm:torus-shape} to continue to apply, one needs to impose that if $A$ intersects the image of such a section, then the entire image is contained in $A$.
The proposed generalization is then straightforward; since we have no need for the precise statements, the details are omitted. \qed
\end{rmk}

\begin{exa}
Consider $(\R^{2 n} - W) \cup H_\epsilon$ with the standard symplectic structure on $\R^{2 n}$, where $W$ is the wall $\{ y_1 = 0 \}$, and $H_\epsilon = \{ x \in \R^{2 n} | \| x \| < \epsilon \}$ (so that $W - H_\epsilon$ is a wall in $\R^{2 n}$ with a hole).
Then $I ((\R^{2 n} - W) \cup H_\epsilon, \omega_0, 0) = \R^n = I (\R^{2 n}, \omega_0, 0)$, but there exists no symplectic diffeomorphism between these symplectic manifolds (a symplectic camel cannot fit through the wall, see e.g.\ \cite[pages 32-33]{mcduff:ist98}). \qed
\end{exa}

These examples lead to a better understanding of the proper formulation of the question posed in the opening paragraph of this section.
If $W_1$ and $W_2$ are two exact symplectic manifolds, and $\Phi \colon H^1 (W_2, \R) \to H^1 (W_1, \R)$ is an isomorphism so that $I (W_1, L, \tau) = I (W_2, L, \tau \circ \Phi)$ for an appropriate closed manifold $L$ and homomorphism $\tau$ (or all $L$ and $\tau$), must there exist a symplectic diffeomorphism $W_1 \to W_2$?
The shape alone may not be able to detect if two manifolds (or subsets of a given manifold) are diffeomorphic, so in general, that may need to be assumed.
It also seems reasonable to assume that the isomorphism $\Phi$ is induced by a diffeomorphism $W_1 \to W_2$.
A more restrictive version of the above question can then be phrased as follows: if $I (W, \omega_1, L, \tau) = I (W, \omega_2, L, \tau)$, must there exist a diffeomorphism $\varphi$ of $W$ so that $\varphi^* \omega_2 = \omega_1$?
For example, does $I (\R^{2 n}, \omega, 0) = \R^n - \{ 0 \}$ imply that $\omega$ is diffeomorphic to the standard symplectic form on $\R^{2 n}$?

At present, with the exception of the special cases discussed above, this question is open.
In each of the examples, the desired symplectic diffeomorphism arises from a corresponding linearized problem.
For the plane, it is the time-one map of a time-dependent vector field, and for cotangent bundles, it comes from a linear map on $H^1 (T^n, \Z)$.
In the latter case, it also sends a distinguished foliation by Lagrangian submanifolds (diffeomorphic to $T^n$) to another such distinguished foliation.

These questions can be translated to the contact case in a straightforward manner.
Suppose that $M_1$ and $M_2$ are two coorientable contact manifolds, and $\Phi \colon PH^1 (M_2, \R) \to PH^1 (M_1, \R)$ is an isomorphism so that the (modified) contact shapes satisfy $I_C (M_1, L, \tau) = I_C (M_2, L, \tau \circ \Phi)$ (or $\It_C (M_1, L, \tau) = \It_C (M_2, L, \tau \circ \Phi)$) for an appropriate closed manifold $L$ and homomorphism $\tau$ (or all $L$ and $\tau$, and similarly if $M_1$ and $M_2$ are diffeomorphic and $\Phi$ is induced by a diffeomorphism).
Then must there exist a contact diffeomorphism $M_1 \to M_2$?
With the exception of the following two results, this question is also open.

Consider the unit sphere $S^3 \subset \mathbb C^2$ with its standard contact structure, and for $0 < r < 1$, denote by $U_r$ the open solid torus $\{ (z_1, z_2) \in \mathbb C^2 \mid | z_1 | < r \} \cap S^3$.
The next theorem is the main result of \cite{eliashberg:nio91}; its proof is based on the shape invariant.
Since the precise argument is not needed here, it is not repeated.

\begin{thm}[\cite{eliashberg:nio91}]
There exists a contact diffeomorphism between $U_{r_1}$ and $U_{r_2}$ if and only if the difference $(1 / r_1^2) - (1 / r_2^2)$ is an integer.
\end{thm}

The second theorem is a straightforward adaptation of Theorem~\ref{thm:cotangent-symp-rigidity} to unit cotangent bundles.

\begin{thm}
Let $U \subset U'$ and $V \subset V'$ be open and connected subsets of $S^{n - 1}$ so that $H^1 (U', \Z)$ and $H^1 (V', \Z)$ are trivial.
Then the following statements are equivalent:
\begin{enumerate}
\item there exists a contact diffeomorphism $\varphi \colon T^n \times U' \to T^n \times V'$ that maps $T^n \times U$ to $T^n \times V$,
\item there exists a diffeomorphism $\varphi \colon T^n \times U' \to T^n \times V'$ that maps $T^n \times U$ to $T^n \times V$, and $I_C (T^n \times U, \iota_a^*) = I_C (T^n \times V, \iota_a^* \circ \varphi^*)$, where $a \in U$ is any point,
\item there exists a diffeomorphism $\varphi \colon T^n \times U' \to T^n \times V'$ that maps $T^n \times U$ to $T^n \times V$, and $\It_C (T^n \times U, \iota_a^*) = \It_C (T^n \times V, \iota_a^* \circ \varphi^*)$, where $a \in U$ is any point, and
\item there exists a matrix $A \in O (n, \Z)$ so that $V = A U$.
\end{enumerate}
In fact, the matrix $A$ represents the homomorphism $PH^1 (T^n, \Z) \to PH^1 (T^n, \Z)$ induced by $\varphi$, and given (4), we may choose $U' = V' = S^{n - 1}$, and there exists a contact diffeomorphism that preserves the canonical contact form $\alpha_\can$.
\end{thm}

\begin{proof}
The proof only requires minor modifications of the proof of Theorem~\ref{thm:cotangent-symp-rigidity}.
The implication (1) implies (3) is again obvious.
In case (3), by the same argument the hypotheses yield an isomorphism $\Phi \colon H^1 (T^n, \Z) \to H^1 (T^n, \Z)$ (which in general may not preserve length) so that $\iota_a^* \circ \varphi^* = \Phi \circ \iota_{\varphi (a)}^*$, and induces a well-defined isomorphism $P\Phi \colon PH^1 (T^n, \Z) \to PH^1 (T^n, \Z)$ so that $U = P\Phi (V)$ (recall that there is no freedom of translation in the contact case).
The definition of the contact diffeomorphism in (1) given statement (4) is the same as before (with $b = 0$), and thus it extends to a (strictly) contact diffeomorphism $T^n \times S^{n - 1} \to T^n \times S^{n - 1}$ (the fact that $H^1 (S^1, \Z)$ is non-zero is irrelevant in the case $n = 2$).
The proof that (1) and (4) are equivalent to (2) is verbatim the same.
\end{proof}

\section{Homeomorphisms that preserve shape} \label{sec:homeos}
In this section we extend the definition of shape preserving to homeomorphisms and derive a few basic properties of such homeomorphisms.
Definition~\ref{dfn:shape-preserving-non-exact} in fact extend verbatim to homeomorphisms of symplectic and contact manifolds, since it does not involve derivatives.
For convenience, we restate the definition here.
In order to not have to duplicate every statement, we again write $M_1$ and $M_2$ for either symplectic or contact manifolds.

\begin{dfn} \label{dfn:homeo-shape-preserving}
A homeomorphism $\varphi \colon M_1 \to M_2$ preserves the (modified) shape of two exact open subsets $U \subset M_1$ and $V \subset M_2$ such that $\overline{U} \subset M_1$ is compact and $\varphi (\overline{U}) \subset V$ if $I (U, L, \tau) \subset I (V, L, \tau \circ \varphi^*)$ (or $I_C (U, L, \tau) \subset I_C (V, L, \tau \circ \varphi^*)$ or $\It_C (U, L, \tau) \subset \It_C (V, L, \tau \circ \varphi^*)$) for every $L$ and every homomorphism $\tau$, and it preserves the (modified) shape if it preserves the (modified) shape of all subsets $U \subset M_1$ and $V \subset M_2$ as above. \qed
\end{dfn}

\begin{pro}
Suppose that $\varphi \colon M_1 \to M_2$ is a homeomorphism that preserves the (modified) shape invariant, and in addition that $\varphi$ is a diffeomorphism.
Then $\varphi$ is a symplectic (respectively contact) diffeomorphism.
\end{pro}

\begin{proof}
We give the proof for a homeomorphism of a symplectic manifold.
The proof in the contact case is verbatim the same.

Let $x \in M_1$ and $U \subset M_1$ be a Darboux neighborhood of $x$ that is diffeomorphic to an open ball.
By Theorem~\ref{thm:symp-shape-preserving}, the restriction of $\varphi$ to $U$ is symplectic.
Since $x$ was an arbitrary point in $M_1$, $\varphi$ is symplectic.
\end{proof}

The following proposition is implicitly contained in the proofs of Theorems~\ref{thm:symp-shape-preserving} and \ref{thm:contact-shape-preserving}, but stated here for emphasis.

\begin{pro}
Let $\varphi \colon M_1 \to M_2$ be a homeomorphism that preserves the (modified) shape invariant, and $L \subset M_1$ be an (embedded) Lagrangian (respectively coisotropic) submanifold such that $\varphi (L)$ is smooth.
Suppose that $M_1$ and $M_2$ are open and connected subsets of $\R^k$ (with $k = 2 n$ or $2 n - 1$, respectively), or that $L = T^n$.
Then $\varphi (L)$ is again Lagrangian (respectively not convex).
\end{pro}

\begin{proof}
Arguing by contradiction, assume that $\varphi (L)$ is not Lagrangian (respectively convex).
Then argue as in the proof of Theorem~\ref{thm:symp-shape-preserving} (respectively Theorem~\ref{thm:contact-shape-preserving}) using Theorem~\ref{thm:rig-lag-hom-class} (respectively Theorem~\ref{thm:rig-coiso-hom-class} or Theorem~\ref{thm:rig-coiso-arnold-conjecture}) to conclude that $\varphi$ does not preserve the (modified) shape invariant.
\end{proof}

\begin{rmk}
The preceding theorem can again be extended to arbitrary (closed and connected $n$-dimensional) manifolds $L$ verbatim as in Remark~\ref{rmk:coiso-arnold-conjecture}. \qed
\end{rmk}

\begin{pro} \label{pro:symp-homeo-shape-preserving}
Let $\varphi_k \colon M_1 \to M_2$ be symplectic (respectively contact) diffeomorphisms that converge to a homeomorphism $\varphi \colon M_1 \to M_2$ uniformly on compact subsets.
Then $\varphi$ preserves the (modified) shape invariant.
\end{pro}

\begin{proof}
The proofs of Propositions~\ref{pro:continuous-symp-shape} and \ref{pro:continuous-contact-shape} apply verbatim to show that the limit $\varphi$ preserves the (modified) shape invariant.
\end{proof}

\begin{rmk}
A homeomorphism $\varphi$ (of a symplectic manifold) as in the above proposition is called a symplectic homeomorphism \cite{mueller:ghh07} (see its final section for open manifolds).
The converse implication that a shape preserving homeomorphism is symplectic in the sense of \cite{mueller:ghh07} is not known.
A necessary and sufficient condition for when a given homeomorphism can be approximated uniformly by diffeomorphisms can be found in \cite{mueller:uah14}.
This question as well as a comparison of various other notions of symplectic homeomorphism is work in progress.
The same question for the (modified) contact shape and homeomorphism that can be approximated uniformly by contact diffeomorphisms is also open. \qed
\end{rmk}

\section{Topological Lagrangian submanifolds} \label{sec:lagrangians}
The shape invariant allows us to define what it means for a closed $n$-dimensional topological submanifold to be Lagrangian.
We propose several definitions, and discuss the relationships between them.

Let $L$ as before be a closed and connected smooth $n$-dimensional manifold, and $f \colon L \to W$ be a continuous map.
Assume that there exist tubular neighborhoods $U$ of the zero section $L_0$ in $T^*L$ and $V$ of $f (L)$ in $W$, and a homeomorphism $\varphi \colon U \to V$ such that $\varphi \circ \iota_0 = f$.
We consider the following properties of such a map $f$.
\begin{enumerate}
\item There exists an extension $\varphi$ as above to a symplectic homeomorphism.
\item There exists an extension $\varphi$ as above to a shape preserving homeomorphism.
\item For every tubular neighborhood $N \subset V$ of the image $f (L)$, the shape $I (N, L, \tau)$ is nonempty, where $\tau = f^* \colon H^1 (N, \R) \to H^1 (L, \R)$.
\item For every tubular neighborhood $N_0 \subset V$ of the image $f (L)$, the intersection $\bigcap I (N, L, \tau)$ of shapes is nonempty, where the intersection is over all tubular neighborhoods $N \subset N_0$ of $f (L)$ such that the inclusion $i \colon N \to N_0$ induces an isomorphism on the first cohomology groups and $\tau \circ i^* = f^*$, where $f^* \colon H^1 (N_0, \R) \to H^1 (L, \R)$.
\end{enumerate}

It is obvious that $(1) \Rightarrow (2) \Rightarrow (4) \Rightarrow (3)$.
Indeed, the first implication follows from Proposition~\ref{pro:symp-homeo-shape-preserving}, and the second implication from the definition of shape preserving, while the third implication is the special case $N = N_0$.
Moreover, if $L$ is a smooth submanifold, then each of the four conditions imply that $L$ is Lagrangian.
This follows immediately from Theorem~\ref{thm:rig-lag-hom-class} if $(W, \omega) = (\R^{2 n}, \omega_0)$ or Remark~\ref{rmk:arnold-conjecture} if $L = T^n$.
For arbitrary (closed and connected $n$-dimensional) manifolds $L$, one must again replace the shape invariant above by a more restrictive shape invariant as in Remarks~\ref{rmk:generalized-shape} and \ref{rmk:arnold-conjecture}.
Finally note that if the intersection in (4) contains a single point and $\tau = 0$, one can assign a $\lambda$-period to the map $f \colon L \to W$.

\section{A symplectic capacity from the shape invariant} \label{sec:capacity}
We have seen that in the present context the (symplectic) shape invariants have several advantages over symplectic capacities, see Remarks~\ref{rmk:distinguish-anti-symp} and \ref{rmk:capacities-infinite}.
In this section we observe that a (small) part of the shape invariants defines a symplectic capacity (which is normalized by its value on unit polydisks rather than unit balls).

\begin{dfn}
We define $c (M, \omega)$ as the non-negative number (possibly $\infty$) that assign to each symplectic manifold $(M, \omega)$ the supremum over the positive generators $\gamma$ of rational vectors $z \in H^1 (T^n, \R) = \R^n$ such that there exists a Lagrangian embedding $\iota \colon T^n \hookrightarrow U$ into an exact open subset $U \subset M$ such that the induced homomorphism $\iota^* \colon H^1 (U, \R) \to H^1 (T^n, \R)$ is trivial and $z = [\iota^* \lambda]$, where $\lambda$ is a primitive one-form of $\omega |_U$. \qed
\end{dfn}

\begin{rmk}
In other words, $c (M, \omega)$ is the supremum over all positive generators of vectors that belong to the rational part of the shape $I (M, 0)$.
Since $\tau = 0$, the $\lambda$-periods $[\iota^* \lambda]$, and in particular their rationality, do not depend on the choice of one-form $\lambda$ with $d\lambda = \omega$ on $U \subset M$. \qed
\end{rmk}

\begin{thm}
The number $c (M, \omega)$ is a symplectic capacity.
More precisely, it satisfies the following axioms:
\begin{itemize}
\item (monotonicity) if there exists a symplectic embedding $(M_1, \omega_1) \to (M_2, \omega_2)$ and $\dim M_1 = \dim M_2$, then $c (M_1, \omega_1) \le c (M_2, \omega_2)$,
\item (conformality) $c (M, r \, \omega) = r \, c (M, \omega)$ for any real number $r \not= 0$, and 
\item (non-triviality) $c (B_1^{2 n}, \omega_0) > 0$ and $c (B_1^2 \times \R^{2 n - 2}, \omega_0) < \infty$.
\end{itemize}
Moreover, $c (M, \omega)$ satisfies the normalization axiom
\begin{itemize}
\item (normalization) $c (B_1^2 \times \cdots \times B_1^2, \omega_0) = \pi = c (B_1^2 \times \R^{2 n - 2}, \omega_0)$,
\end{itemize}
where $B_1^2 \times \cdots \times B_1^2 \subset \R^{2 n}$ denotes the polydisk of dimension $2 n$.
\end{thm}

\begin{proof}
The monotonicity axiom follows immediately from Proposition~\ref{pro:symp-shape}, and the conformality axiom is obvious from the definition.
For any real number $0 < \epsilon < 1$, the standard embedding $\iota \colon T^n = S^1 (1 - \epsilon) \times \cdots \times S^1 (1 - \epsilon) \hookrightarrow B_1^2 \times \cdots \times B_1^2$ is Lagrangian, and $[\iota^* \lambda] = (\pi (1 - \epsilon)^2, \ldots, \pi (1 - \epsilon)^2)$ with generator $\pi (1 - \epsilon)^2$, which shows that $c (B_1^2 \times \cdots \times B_1^2, \omega_0) \ge \pi$.
On the other hand, $c (B_1^2 \times \R^{2 n - 2}, \omega_0) \le \pi$ by Sikorav's Theorem~\ref{thm:sikorav}.
\end{proof}

\section*{Acknowledgments}
The recent article \cite{ms:gae14} uses Gromov's alternative and the construction of a diffeomorphism that cannot be approximated uniformly (on compact subsets) by contact diffeomorphisms to prove Corollary~\ref{cor:rig-contact-diff}.
We would like to thank Yasha Eliashberg for explaining to us (at a conference at ETH Z\"urich in June 2013) that his shape invariant can be used to construct such a diffeomorphism on the unit cotangent bundle of a torus.
We would also like to thank both Eliashberg and the anonymous referee of the paper \cite{ms:gae14} for encouraging us to prove a local version of $C^0$-rigidity of contact diffeomorphisms for embeddings, which was the starting point of the present paper.

The results in this paper concerning symplectic embeddings and the symplectic shape invariants were presented at seminars at IAS/Princeton in November 2014, at the University of Illinois at Urbana-Champaign in April 2015, at UGA in November 2015, and at Harvard in February 2017, and the results on contact embeddings and the contact shape invariants were presented at a seminar at Georgia Tech in November 2016; we thank the audiences for their interest.

This paper is dedicated to my sons Leon (4 years) and Lukas (2 years), who have greatly delayed the completion of this work.
Thank you for everything (and for eventually allowing things to return to a new normal).

\bibliography{char-emb}
\bibliographystyle{plain}

\end{document}